\documentclass{amsart}

\title      [Some computations of Frobenius-Schur indicators]
            {Some computations of Frobenius-Schur indicators of the regular representations of Hopf algebras}
\dedicatory {Dedicated to Professor Mitsuhiro Takeuchi in honor of his distinguished career}
\author     [K. Shimizu]{Kenichi Shimizu}
\address    [K. Shimizu]{Institute of Mathematics, University of Tsukuba. Tsukuba, Ibaraki, 305-8571, Japan.}

\usepackage{amsmath}
\usepackage{amssymb}
\usepackage{amscd}
\usepackage{amsthm}

\newtheorem{counter}            {}[section]
\theoremstyle{definition}
\newtheorem{definition}         [counter]{Definition}
\theoremstyle{plain}
\newtheorem{lemma}              [counter]{Lemma}
\newtheorem{proposition}        [counter]{Proposition}
\newtheorem{theorem}            [counter]{Theorem}
\newtheorem{corollary}          [counter]{Corollary}
\newtheorem{question}           [counter]{Question}
\theoremstyle{remark}

\newtheorem{remark}             [counter]{Remark}

\newcommand{\sgn}       {\mathop{\rm sgn}\nolimits}
\newcommand{\lcm}       {\mathop{\rm lcm}\nolimits}
\newcommand{\ord}       {\mathop{\rm ord}\nolimits}
\newcommand{\floor}     [1]{{\left\lfloor{#1}\right\rfloor}}
\newcommand{\id}        {\mathop{\rm id}\nolimits}
\newcommand{\op}        {{\rm op}}
\newcommand{\cop}       {{\rm cop}}

\newcommand{\Hom}       {\mathop{\rm Hom}\nolimits}
\newcommand{\End}       {\mathop{\rm End}\nolimits}
\newcommand{\Ker}       {\mathop{\rm Ker}\nolimits}

\newcommand{\Irr}       {\mathop{\rm Irr}\nolimits}
\newcommand{\Rep}       {\mathop{\rm Rep}\nolimits}
\newcommand{\Vect}      {\mathop{\bf Vec}\nolimits}

\newcommand{\Trace}     {\mathop{\rm Tr}\nolimits}
\newcommand{\ctrace}    {\mathop{\underline{\rm tr}}\nolimits}
\newcommand{\ptrace}    {\mathop{\rm ptr}\nolimits}
\newcommand{\pdim}      {\mathop{\rm pdim}\nolimits}
\newcommand{\FPdim}     {\mathop{\rm FPdim}\nolimits}
\newcommand{\bicross}   {\mathop{\Join}\nolimits}
\newcommand{\Opext}     {\mathop{\rm Opext}\nolimits}
\newcommand{\legendre}  [2]{{\genfrac{(}{)}{1pt}{}{#1}{#2}}}
\newcommand{\sslash}    {{/\!\!/}} 

\newenvironment{enumalph}
{\begin{enumerate}
    \renewcommand{\labelenumi}{\textnormal{\hbox to 1.4em{(\hfil\alph{enumi}\hfil)}}}}
  {\end{enumerate}}

\begin{document}

\begin{abstract}
  We study Frobenius-Schur indicators of the regular representations of finite-di\-men\-sion\-al semi\-simple Hopf algebras, especially group-the\-o\-ret\-i\-cal ones. Those of various Hopf algebras are computed explicitly. In view of our computational results, we formulate the theorem of Frobenius for semisimple Hopf algebras and give some partial results on this problem.
\end{abstract}

\maketitle

\section{Introduction}
\label{sec:introduction}

In \cite{MR1808131}, Linchenko and Montgomery defined Frobenius-Schur indicators of characters of finite-di\-men\-sion\-al semisimple Hopf algebras as follows: For a character $\chi$ of such a Hopf algebra $H$, the $n$-th Frobenius-Schur indicator of $\chi$ is given by
\begin{equation}
  \label{eq:FS-Hopf}
  \nu_n(\chi) = \chi(\Lambda^{[n]})
\end{equation}
where $\Lambda \in H$ is the normalized integral and $(-)^{[n]}$ means the $n$-th Sweedler power.

In this paper, we study and compute Frobenius-Schur indicators $\nu_n(H)$ of (the characters of) the regular representations of finite-dimensional semisimple Hopf algebras $H$. It follows from the work of Ng and Schauenburg \cite{MR2381536} that these numbers have the following remarkable property: $\nu_n(H) = \nu_n(L)$ for every $n$ if the category of $H$-modules and that of $L$-modules are monoidally equivalent (see also \cite{KMN09}). In view of this invariance, studying them is important for the representation theory of Hopf algebras. In fact, Kashina, Montgomery and Ng studied $\nu_n(H)$ and gave some applications to the representation theory in \cite{KMN09}.

For that reasons, it is interesting to compute explicit values of $\nu_n(H)$. There is a formula of $\nu_n(H)$ due to Kashina, Sommerh\"auser and Zhu: In \cite{MR2213320}, they showed that it is equal to the trace of the linear map $h \mapsto S(h^{[n - 1]})$ ($h \in H$) where $S$ is the antipode of $H$. However, because of difficulties of the computation of the Sweedler power, the computation of $\nu_n(H)$ is still difficult.

To compute Frobenius-Schur indicators of the regular representation, we introduce some new methods. Note that Ng and Schauenburg generalized Frobenius-Schur indicators to objects of linear pivotal categories in \cite{MR2381536}. By using their definition, for a pivotal fusion category $\mathcal{C}$, we define its $n$-th indicator $\nu_n(\mathcal{C})$ in Section~\ref{sec:regular} so that $\nu_n(\mathcal{C}) = \nu_n(H)$ when $\mathcal{C}$ is the category $\Rep(H)$ of finite-dimensional representations of $H$. It turns out that $\nu_n(\mathcal{C})$ is an invariant of fusion categories admitting pivotal structures. By using some categorical methods, we introduce a formula for indicators of group-theoretical categories, which are a well-studied class of fusion categories. Since many known semisimple Hopf algebras are group-theoretical, this formula can be applied to compute $\nu_n(H)$ for various $H$.

As is well-known in the theory of finite groups, for a finite group $G$,
\begin{equation*}
  \nu_n(\mathbb{C}G) = \# \{ g \in G \mid g^n = 1 \}.
\end{equation*}
Frobenius showed that the right-hand side of this equation is divisible by $n$ if $n$ divides $|G|$. It is natural to ask whether $\nu_n(H)$ has the same property (see Definition~\ref{def:frobenius} for a precise statement). Our computational results, which are obtained in Section~\ref{sec:GT-example}, seem to suggest that the answer to this question is ``yes''. We cannot provide a complete answer to this question, but give some partial results on it. 

This paper is organized as follows: In Section~\ref{sec:preliminaries}, we summarize some results on fusion categories and provide some lemmas. In Section~\ref{sec:regular}, we define the $n$-th indicator $\nu_n(\mathcal{C})$ of a pivotal fusion category $\mathcal{C}$ and introduce some basic properties of $\nu_n(\mathcal{C})$. One of the most important property is that $\nu_n(\mathcal{C}) = \nu_n(\mathcal{D})$ for every $n$ if $\mathcal{C}$ and $\mathcal{D}$ are monoidally equivalent (Theorem~\ref{thm:FS-invariance}). We also show that $\nu_n(\mathcal{C})$ is a cyclotomic integer and that $\nu_n(\mathcal{C})$ is real if $\mathcal{C}$ admits a braiding.

In Section~\ref{sec:group-theoretical}, we introduce a formula for indicators of group-theoretical categories. By using this formula, we argue arithmetic properties of indicators of group-theoretical categories. In the following Section~\ref{sec:GT-example}, we compute Frobenius-Schur indicators of the regular representations of various Hopf algebras. In Section~\ref{sec:frobenius}, we formulate and discuss Frobenius theorem for semisimple Hopf algebras, in view of results of Section~\ref{sec:group-theoretical} and Section~\ref{sec:GT-example}.

\section{Preliminaries}
\label{sec:preliminaries}

\subsection{Pivotal categories}

Throughout this paper, the basic theory of Hopf algebras and monoidal categories will be freely used. Our main references are \cite{MR1797619}, \cite{MR2183279} and \cite{MR1321145}. We work over the field $\mathbb{C}$ of complex numbers. Unless otherwise noted, quasi-Hopf algebras are assumed to be finite-dimensional (over $\mathbb{C}$). Functors between $\mathbb{C}$-linear categories are always assumed to be $\mathbb{C}$-linear.

First we fix some conventions. In a monoidal category $\mathcal{C}$, the associativity isomorphism is denoted by $a_{X,Y,Z}: (X \otimes Y) \otimes Z \to X \otimes (Y \otimes Z)$. We may assume that the unit object $\mathbf{1} \in \mathcal{C}$ is strictly unital, that is, it satisfies $\mathbf{1} \otimes V = V = V \otimes \mathbf{1}$ and the unit constraints are identities. A left dual object of $V \in \mathcal{C}$ is denoted by $V^*$ if it exists. Duality morphisms are usually denoted by
\begin{equation*}
  b_V: {\bf 1} \to V \otimes V^* \quad \text{and} \quad
  d_V: V^* \otimes V \to {\bf 1}.
\end{equation*}
We say that $\mathcal{C}$ is left rigid if every object of $\mathcal{C}$ has a left dual. If $\mathcal{C}$ is left rigid, the assignment $V \mapsto V^*$ extends naturally to a contravariant endofunctor $(-)^*$ on $\mathcal{C}$, and $(-)^{**}$ is naturally a monoidal endofunctor on $\mathcal{C}$.

The {\em trace} of a morphism $f: V \to V^{**}$ in $\mathcal{C}$ is given by
\begin{equation*}
  \ctrace(f) = d_{V^*} \circ (f \otimes \id_{V^*}) \circ b_V \in \End_{\mathcal{C}}(\mathbf{1}).
\end{equation*}
Let $g: W \to W^{**}$ be another morphism in $\mathcal{C}$. We can regard $f \otimes g$ as a morphism $V \otimes W \to (V \otimes W)^{**}$ via the canonical isomorphism. It is well-known that
\begin{equation}
  \label{eq:ctrace}
  \ctrace(f \otimes g) = \ctrace(f) \circ \ctrace(g) \in \End_\mathcal{C}(\mathbf{1})
\end{equation}
if $\id_V \otimes \ctrace(g) = \ctrace(g) \otimes \id_V$. This condition is satisfied, for example, if $\mathcal{C}$ is a fusion category which will be mentioned later.

A {\em pivotal structure} on a left rigid monoidal category $\mathcal{C}$ is an isomorphism $\id_\mathcal{C} \to (-)^{**}$ of monoidal functors. Let $j$ be a pivotal structure on $\mathcal{C}$. It is known that $j_V^* = j_{V^*}^{-1}$ for every $V \in \mathcal{C}$, see \cite[Appendix~A]{MR2095575}. The {\em pivotal trace} of $f: V \to V$ with respect to $j$ is given and denoted by $\ptrace_j(f) = \ctrace(j_V \circ f)$. The pivotal trace of $\id_V$ is called the {\em pivotal dimension} of $V$ and denoted by $\pdim_j(V)$. The subscript $j$ of $\ptrace_j$ and of $\pdim_j$ are often omitted if $j$ is understood.

A {\em pivotal category} is a left rigid monoidal category equipped with a pivotal structure. Let $F: \mathcal{C} \to \mathcal{D}$ be a monoidal functor between pivotal categories that preserves the pivotal structures and the unit objects. Then,
\begin{equation}
  \label{eq:ptrace-1}
  \ptrace(F(f)) = F(\ptrace(f))
\end{equation}
for every morphism $f: V \to V$ in $\mathcal{C}$. See, for example, \cite[\S 6]{MR2381536} for the proof.

For a monoidal category $\mathcal{C}$, we denote by $\mathcal{C}^{\rm rev}$ the monoidal category with the underlying category $\mathcal{C}$ and the reversed tensor product $\otimes^{\rm rev}$ given by $V \otimes^{\rm rev} W = W \otimes V$. If $\mathcal{C}$ is a pivotal category, $\mathcal{C}^{\rm rev}$ is naturally a pivotal category. We denote by $V^{\rm rev}$ the object $V$ of $\mathcal{C}$ regarded as an object of $\mathcal{C}^{\rm rev}$. For a morphism $f$ in $\mathcal{C}$, $f^{\rm rev}$ has the same meaning. For $f: V \to V$ in $\mathcal{C}$, we have
\begin{equation}
  \label{eq:ptrace-2}
  \ptrace(f^{\rm rev}) = \ptrace(f^*).
\end{equation}
Note that, in general, the pivotal trace of $f$ and of $f^*$ are different. $\mathcal{C}$ is said to be {\em spherical} if $\ptrace(f) = \ptrace(f^*)$ for every $f: V \to V$ in $\mathcal{C}$, see \cite[Definition~2.5]{MR1686423}.

\subsection{Fusion categories}

A {\em fusion category} $\mathcal{C}$ is a linear semisimple rigid monoidal category with finitely many isomorphism classes of simple objects such that the unit object is simple and $\End_\mathcal{C}(V) \cong \mathbb{C}$ for every simple object of $\mathcal{C}$. We denote by $\Irr(\mathcal{C})$ the set of (representatives of) isomorphism classes of simple objects. The Grothendieck ring of $\mathcal{C}$ is denoted by $K_\mathbb{Z}(\mathcal{C})$. Set $K(\mathcal{C}) = K_{\mathbb{Z}}(\mathcal{C}) \otimes_{\mathbb{Z}} \mathbb{C}$.

Let $\mathcal{C}$ be a fusion category. Then, since $\End_\mathcal{C}({\bf 1}) \cong \mathbb{C}$ by definition, we may treat the trace of a morphism as a complex number. We first prove that pivotal dimensions are cyclotomic integers (cf. Corollary~5.15 and Corollary~8.54 of \cite{MR2183279}).

\begin{lemma}
  \label{lem:cyclo-pdim}
  Let $\mathcal{C}$ be a pivotal fusion category. There exists a root of unity $\xi$ such that $\pdim(V) \in \mathbb{Z}[\xi]$ for every $V \in \mathcal{C}$.
\end{lemma}
\begin{proof}
  Equation~(\ref{eq:ctrace}) yields that the assignment $V \mapsto \pdim(V)$ defines a representation of $K_\mathbb{Z}(\mathcal{C})$. It is obvious that $\pdim(V)$ is an algebraic integer. Now the result follows from \cite[Corollary~8.53]{MR2183279}.
\end{proof}

It is conjectured, but not proved, that every fusion category admits a pivotal structure \cite[Conjecture~2.8]{MR2183279}. Instead, we can always define the {\em Frobenius-Perron dimensions} \cite[Section~8]{MR2183279} of objects of fusion categories; That of $V \in \mathcal{C}$ is defined to be the Frobenius-Perron eigenvalue of the left multiplication of $V$ on $K(\mathcal{C})$ and is denoted by $\FPdim(V)$.

Under certain assumptions, a fusion category admits a pivotal structure. Recall that $V^{**} \cong V$ for every $V \in \mathcal{C}$. For a simple object $V$, fix an isomorphism $g: V \to V^{**}$ and set $|V|^2 := \ctrace(g) \ctrace((g^*)^{-1})$. This does not depends on the choice of $g$ and is called the {\em squared norm} of $V$. A fusion category $\mathcal{C}$ is said to be {\em pseudo-unitary} if
\begin{equation*}
  \sum_{V \in \Irr(\mathcal{C})} |V|^2 = \sum_{V \in \Irr(\mathcal{C})} \FPdim(V)^2.
\end{equation*}
If $\mathcal{C}$ is pseudo-unitary, $\mathcal{C}$ has a canonical pivotal structure $j$ such that $\pdim_j(V) = \FPdim(V)$ for every $V \in \mathcal{C}$ \cite[Proposition~8.23]{MR2183279}. Every monoidal equivalence between pseudo-unitary fusion categories preserves the canonical pivotal structures \cite[Corollary~6.2]{MR2381536}. Note that the category $\Rep(H)$ of finite-dimensional representations of a semisimple quasi-Hopf algebra $H$ is pseudo-unitary \cite[Proposition~8.24]{MR2183279}. Hence it has a canonical pivotal structure $j$ such that
\begin{equation}
  \label{eq:dimensions}
  \pdim_j(V) = \FPdim(V) = \dim_{\mathbb{C}}(V)
\end{equation}
for every $V \in \Rep(H)$.

\subsection{Ribbon categories}

Let $\mathcal{C}$ be a left rigid braided monoidal category with braiding $c$. The {\em Drinfeld isomorphism} $u: \id_\mathcal{C} \to (-)^{**}$ is defined by
\begin{equation*}
  u_V = (d_V \otimes \id_{V^{**}}) a_{V^*,V,V^{**}}^{-1}
  (\id_{V^*} \otimes c_{V, V^{**}}^{-1}) a_{V^*, V^{**}, V}^{} (b_{V^*} \otimes \id_V).
\end{equation*}
This satisfies $u_{V \otimes W} = (u_V \otimes u_W) c_{V,W}^{-1} c_{W,V}^{-1}$ for all $V, W \in \mathcal{C}$. A {\em ribbon category} is a left rigid braided monoidal category $\mathcal{C}$ equipped with a {\em twist} \cite[Definition~XIV.3.2]{MR1321145}, that is, an isomorphism $\theta: \id_{\mathcal{C}} \to \id_{\mathcal{C}}$ of functors satisfying
\begin{equation*}
  \theta_{V \otimes W} = (\theta_V \otimes \theta_W) c_{W, V}^{} c_{V, W}^{}
  \quad \text{and} \quad (\theta_V)^* = \theta_{V^*}.
\end{equation*}
The former equation holds if and only if $j = u\theta: \id_\mathcal{C} \to (-)^{**}$ is a pivotal structure on $\mathcal{C}$. If $\theta$ is a twist, this $j$ is a spherical pivotal structure on $\mathcal{C}$. Under the assumption that $\mathcal{C}$ is a fusion category, $j$ is spherical if and only if $\theta$ is a twist.

We denote by $\mathcal{Z}(\mathcal{C})$ the {\em left} Drinfeld center of a monoidal category $\mathcal{C}$. Recall that the objects of $\mathcal{Z}(\mathcal{C})$ are pairs $(V, e_V)$ of an object $V \in \mathcal{C}$ and a half-braiding $e_V: V \otimes (-) \to (-) \otimes V$. M\"uger \cite{MR1966525} showed that $\mathcal{Z}(\mathcal{C})$ is a fusion category if so is $\mathcal{C}$, and then the forgetful functor $\Pi_{\mathcal{C}}: (V, e_V) \mapsto V$ has a two-sided adjoint $I_{\mathcal{C}}: \mathcal{C} \to \mathcal{Z}(\mathcal{C})$. On the objects, we have
\begin{equation}
  \label{eq:adjoint}
  I_{\mathcal{C}}(V) \cong \bigoplus_{(X, e_X) \in \Irr(\mathcal{Z}(\mathcal{C}))} (X, e_X)^{\oplus [X:V]}
\end{equation}
where $[X:V] := \dim_{\mathbb{C}} \Hom_{\mathcal{C}}(X, V)$.

A pivotal structure $j$ on $\mathcal{C}$ naturally induces a pivotal structure $J$ on $\mathcal{Z}(\mathcal{C})$ such that $\ptrace_{J}(f) = \ptrace_j(\Pi_{\mathcal{C}}(f))$ for every endomorphism $f$ in $\mathcal{Z}(\mathcal{C})$. If $j$ is spherical, since $\Pi_{\mathcal{C}}(f^*) = f^*$ for every morphism $f$ in $\mathcal{C}$, also $J$ is spherical. Therefore, the Drinfeld center of a spherical fusion category has a canonical twist, and hence it turns into a ribbon category.

\subsection{Frobenius-Schur indicators}

Frobenius-Schur indicators for pivotal categories are introduced by Ng and Schauenburg in \cite{MR2381536}. Let us briefly recall the definition. First, let $\mathcal{C}$ be a monoidal category with left duality. Then there exist isomorphisms
\begin{equation*}
  A^X_{Y,Z}: \Hom_{\mathcal{C}}(X, Y \otimes Z) \to \Hom_{\mathcal{C}}(Y^* \otimes X, Z)
\end{equation*}
and
\begin{equation*}
  B^{X,Y}_Z: \Hom_{\mathcal{C}}(X \otimes Y, Z) \to \Hom_{\mathcal{C}}(X, Z \otimes Y^*)
\end{equation*}
that are natural in $X$, $Y$ and $Z$ \cite[Proposition~XIV.2.2]{MR1321145}. We set
\begin{equation*}
  D_{V,W} = B^{\mathbf{1}, V^*}_{W} \circ A^{\mathbf{1}}_{V,W}:
  \Hom_{\mathcal{C}}(\mathbf{1}, V \otimes W) \to \Hom_{\mathcal{C}}(\mathbf{1}, W \otimes V^{**})
\end{equation*}
for $V, W \in \mathcal{C}$. For an object $V \in \mathcal{C}$, define $V^{\otimes n} \in \mathcal{C}$ inductively by $V^{\otimes 0} = \mathbf{1}$, $V^{\otimes 1} = V$ and $V^{\otimes n} = V \otimes V^{\otimes (n-1)}$ for $n \ge 2$. Now we assume $\mathcal{C}$ to be a linear pivotal category (with finite-dimensional $\Hom$-spaces). Then linear automorphisms
\begin{equation*}
  E_V^{(n)}: \Hom_{\mathcal{C}}(\mathbf{1}, V^{\otimes n}) \to \Hom_{\mathcal{C}}(\mathbf{1}, V^{\otimes n})
  \quad (n = 1, 2, \cdots)
\end{equation*}
are defined by
\begin{equation*}
  E_V^{(n)}(f) = a^{(n)}_V \circ \left(\id_{V^{\otimes(n-1)}} \otimes j_V^{-1} \right) \circ D_{V, V^{\otimes (n - 1)}}(f),
\end{equation*}
where $a_V^{(n)}: V^{\otimes (n - 1)} \otimes V \to V^{\otimes n}$ is the unique isomorphism obtained by composing tensor products of associativity and identity morphisms. Explicitly, it is given inductively by
\begin{equation}
  \label{eq:FS-assoc}
  a_V^{(1)} = \id_V, \quad
  a_V^{(n)} = (\id_V \otimes a_V^{(n - 1)}) \circ a_{V, V^{\otimes (n - 2)}, V}
  \quad (n \ge 2).
\end{equation}
The {\em $n$-th Frobenius-Schur indicator of $V$} is given and denoted by
\begin{equation}
  \label{eq:FS-pivotal}
  \nu_n(V) = \Trace \left(E_V^{(n)}\right)
\end{equation}
where $\Trace$ means the usual trace of linear maps.

Ng and Schauenburg showed that the above formula is a generalization of~(\ref{eq:FS-Hopf}), that is, the right-hand side of~(\ref{eq:FS-pivotal}) is equal to the right-hand side of~(\ref{eq:FS-Hopf}) if $V \in \Rep(H)$ for some semisimple Hopf algebra $H$ (see \cite[Remark~4.3]{MR2366965}). They also generalized the ``third formula'' of Kashina, Sommerh\"auser and Zhu \cite[\S 6.4, Corollary]{MR2213320} to objects of spherical fusion category: If $\mathcal{C}$ is a spherical fusion category, we have
\begin{equation}
  \label{eq:FS-third}
  \nu_n(V) = \frac{1}{\dim(\mathcal{C})} \ptrace\left(\theta^n_{I_{\mathcal{C}}(V)}\right) \quad (V \in \mathcal{C})
\end{equation}
where $\theta$ is the canonical twist of $\mathcal{Z}(\mathcal{C})$ and $\dim(\mathcal{C}) = \sum_{V \in \Irr(\mathcal{C})} |V|^2$ (see \cite[Theorem~4.1]{MR2313527}). This formula plays an important role in Section~\ref{sec:regular}.

One might think that the definition of $E_V^{(n)}$ is slight complicated. It should be understood in terms of graphical calculus \cite[\S 5]{MR2381536}. If we did so, it would be obvious that the $n$-th power of $E_V^{(n)}$ is equal to the identity. The following lemma follows immediately this observation (see \cite[Theorem~5.1]{MR2381536}).

\begin{lemma}
  \label{lem:cyclo-FS}
  $\nu_n(V) \in \mathbb{Z}[\zeta_n]$ where $\zeta_n$ is a primitive $n$-th root of unity.
\end{lemma}

The definition of Frobenius-Schur indicators depends on the choice of pivotal structures. Let $\mathcal{C}$ be a pivotal fusion category with pivotal structure $j$ and let $j'$ be another pivotal structure on $\mathcal{C}$. For a while, we denote by $\nu'_n$ the $n$-th Frobenius-Schur indicator with respect to $j'$. If $V$ is a simple object of $\mathcal{C}$, then there exists $\lambda_V \in \mathbb{C}^{\times}$ such that $j'_V = \lambda_V j_V$. As remarked in \cite[Remark~5.2]{MR2313527}, we have
\begin{equation*}
  \nu'_n(V) = \lambda_V^{-1} \nu_n(V).
\end{equation*}
and, on the other hand, $\pdim_{j'}(V) = \lambda_V \pdim_{j}(V)$. Hence we have the following:

\begin{lemma}
  \label{lem:FS-invariance}
  Let $\mathcal{C}$ be a fusion category that admits a pivotal structure. If $V \in \mathcal{C}$ is a simple object, the value $\widetilde{\nu}_n(V) = \nu_n(V) \pdim(V)$ does not depend on the choice of pivotal structures on $\mathcal{C}$.
\end{lemma}

\begin{remark}
  \label{rem:FS-invariance}
  Let $\mathcal{C}$ be a linear monoidal category with left duality. Assume that $V \in \mathcal{C}$ is isomorphic to $V^{**}$. Then, we fix an isomorphism $g: V \to V^{**}$ and define linear automorphisms
  \begin{equation*}
    E_{V, g}^{(n)}: \Hom_{\mathcal{C}}(\mathbf{1}, V^{\otimes n}) \to \Hom_{\mathcal{C}}(\mathbf{1}, V^{\otimes n})
    \quad (n = 1, 2, \cdots)
  \end{equation*}
  by the same formula as $E_V^{(n)}$ but with $j_V$ replaced by $g$. We define $\widetilde{\nu}_n(V)$ by
  \begin{equation*}
    \widetilde{\nu}_n(V) = \Trace(E_{V, g}^{(n)}) \ctrace(g).
  \end{equation*}
  analogously. If $\End_{\mathcal{C}}(V) \cong \mathbb{C}$, the right-hand side does not depends on the choice of $g: V \to V^{**}$. In particular, if $V$ is a simple object of a fusion category, $\widetilde{\nu}_n(V)$ can be well-defined.
\end{remark}

\subsection{Indicators of invertible objects}

An object $V$ of a monoidal category $\mathcal{C}$ is said to be {\em invertible} if the functor $V \otimes (-)$ is an equivalence. In this subsection, we introduce some formulae for Frobenius-Schur indicators of invertible objects.

Let $V$ be an invertible object of a pivotal fusion category $\mathcal{C}$. Then, by induction on $n$, we see that $V^{\otimes n}$ is a simple object for every $n$. Therefore,
\begin{equation*}
  \dim_{\mathbb{C}} \Hom_{\mathcal{C}}(\mathbf{1}, V^{\otimes n}) =
  \begin{cases}
    1 & \text{if $V^{\otimes n} \cong \mathbf{1}$}, \\
    0 & \text{otherwise}.
  \end{cases}
\end{equation*}
In particular, $\nu_n(V) = 0$ unless $V^{\otimes n} \cong \mathbf{1}$. If $V^{\otimes n} \cong \mathbf{1}$, we have
\begin{equation*}
  \Hom_{\mathcal{C}}(V^*, V^{\otimes (n - 1)}) \cong \Hom_{\mathcal{C}}(\mathbf{1}, V^{\otimes n}) \cong \Hom_{\mathcal{C}}(\mathbf{1}, \mathbf{1}) \cong \mathbb{C}.
\end{equation*}
Since $V^*$ is simple, this implies that $V^*$ is isomorphic to $V^{\otimes (n - 1)}$.

\begin{theorem}
  \label{thm:FS-invertible}
  Let $\mathcal{C}$ be a pivotal fusion category and $V \in \mathcal{C}$ an invertible object such that $V^{\otimes n} \cong \mathbf{1}$. Fix an isomorphism $\lambda: V^* \to V^{\otimes (n - 1)}$. Then, we have
  \begin{equation*}
    \nu_n(V) = \pdim(V^*) \cdot
    \left(d_V (\lambda^{-1} \otimes \id_V) \circ (a_V^{(n)})^{-1} \circ (\id_V \otimes \lambda) b_V: \mathbf{1} \to \mathbf{1} \right)^{-1}
  \end{equation*}
  where $a_V^{(n)}: V^{\otimes (n - 1)} \otimes V \to V^{\otimes n}$ is the isomorphism given by~(\ref{eq:FS-assoc}).
\end{theorem}
\begin{proof}
  Set $p = d_V (\lambda^{-1} \otimes \id_V) (a_V^{(n)})^{-1}: V \otimes X \to \mathbf{1}$ and $q = (\id_V \otimes \lambda) b_V: \mathbf{1} \to V \otimes X$ where $X = V^{\otimes (n - 1)}$. One can easily see that $p$ and $q$ are isomorphisms, and hence so is $p \circ q$. Consider the commutative diagram
  \begin{equation*}
    \begin{CD}
      \Hom_{\mathcal{C}}(\mathbf{1}, V \otimes V^*)
      @>{\Hom_{\mathcal{C}}(\mathbf{1}, \id_V \otimes \lambda)}>>
      \Hom_{\mathcal{C}}(\mathbf{1}, V \otimes X) \\
      @V{D_{V,V^*}}VV @VV{D_{V,X}}V \\
      \Hom_{\mathcal{C}}(\mathbf{1}, V^* \otimes V^{**})
      @>>{\Hom_{\mathcal{C}}(\mathbf{1}, \lambda \otimes \id_{V^{**}})}>
      \Hom_{\mathcal{C}}(\mathbf{1}, X \otimes V^{**}).
    \end{CD}
  \end{equation*}
  By chasing $b_V$ in the diagram, we have $D_{V,X}(q) = (\lambda \otimes \id_{V^{**}}) \circ b_{V^*}$. Hence
  \begin{equation*}
    p \circ E_V^{(n)}(q) = d_V(\lambda^{-1} \otimes j_V^{-1})D_{V,X}(q) = d_V(\id_{V^*} \otimes j_V^{-1}) b_V = \pdim(V^*).
  \end{equation*}
  Since $\dim_{\mathbb{C}} \Hom_{\mathcal{C}}(\mathbf{1}, V^{\otimes n}) = 1$, we have $p \circ E_V^{(n)}(q) = \nu_n(V) \cdot p \circ q$. This implies that $\nu_n(V) = \pdim(V^*) \cdot (p \circ q)^{-1}$.
\end{proof}

The following is an immediate consequence of Theorem~\ref{thm:FS-invertible}. We use the same notations as in \cite[XV.5]{MR1321145} for quasi-Hopf algebras.

\begin{corollary}
  \label{cor:FS-invertible-1}
  Let $H = (H, \Delta, \varepsilon, \Phi, S, \alpha, \beta)$ be a semisimple quasi-Hopf algebra and $V$ a one-di\-men\-sion\-al representation of $H$ with character $\chi$. Then we have
  \begin{equation*}
    \nu_n(V) = \frac{\delta_{\chi^n, \varepsilon}}{\chi(\alpha)\chi(\beta)}
    \prod_{k = 1}^{n - 2} \langle \chi \otimes \chi^k \otimes \chi, \Phi \rangle
  \end{equation*}
  with respect to the canonical pivotal structure on $\Rep(H)$.
\end{corollary}

For a finite group $\Gamma$ and a normalized 3-cocycle $\omega \in Z^3(\Gamma, \mathbb{C}^\times)$, let $\Vect_\omega^\Gamma$ denote the category of finite-dimensional $\Gamma$-graded vector spaces with associator given by $\omega$. This category is monoidally equivalent to $\Rep(\mathbb{C}^\Gamma_\omega)$ where $\mathbb{C}^\Gamma_\omega$ is a semisimple quasi-Hopf algebra defined as follows: $\mathbb{C}^\Gamma_\omega$ has same algebra structure, comultiplication, counit and antipode as $\mathbb{C}^\Gamma$, the dual of Hopf algebra of $\mathbb{C}\Gamma$. $\Phi$, $\alpha$ and $\beta$ are given respectively by
\begin{equation*}
  \Phi = \sum_{x, y, z \in \Gamma} \omega(x, y, z) e_x \otimes e_y \otimes e_z, \quad
  \alpha = 1 \quad \text{and} \quad
  \beta = \sum_{g \in \Gamma} \omega(g, g^{-1}, g)^{-1} e_g
\end{equation*}
where $\{ e_g \}_{g \in \Gamma}$ is the dual basis of $\{ g \}_{g \in \Gamma}$. For $g \in \Gamma$, let $E_g$ denote a one-di\-men\-sion\-al $\Gamma$-graded vector space having a non-zero $g$-component. By Corollary~\ref{cor:FS-invertible-1}, we have the following formula due to Ng and Schauenburg \cite[Proposition~7.1]{MR2366965}.

\begin{corollary}
  \label{cor:FS-invertible-2}
  With respect to the canonical pivotal structure, we have
  \begin{equation*}
    \nu_n(E_g) = \delta_{g^n, 1} \prod_{k = 1}^{n - 1} \omega(g, g^k, g).
  \end{equation*}
\end{corollary}

\section{Indicators of the regular representation}
\label{sec:regular}

Let $\mathcal{C}$ be a pivotal fusion category. For a positive integer $n$, set
\begin{equation}
  \label{eq:FS-def}
  \nu_n(\mathcal{C}) = \sum_{V \in \Irr(\mathcal{C})} \nu_n(V) \pdim(V).
\end{equation}
As we remarked in Remark~\ref{rem:FS-invariance}, the right-hand side can be defined for arbitrary fusion category $\mathcal{C}$. However, we always assume that $\mathcal{C}$ has a pivotal structure since examples of non-pivotal categories are not known. In fact, examples we deal with in this paper are all spherical.

If $\mathcal{C}$ is $\Rep(H)$ for some semisimple quasi-Hopf algebra $H$ (with the canonical pivotal structure), then $\nu_n(\mathcal{C})$ is equal to $\nu_n(H)$, the $n$-th Frobenius-Schur indicator of the regular representation of $H$. Indeed, it follows from Artin-Wedderburn theorem and the additivity of Frobenius-Schur indicators \cite[Corollary~7.8]{MR2381536} that
\begin{equation*}
  \nu_n(H) = \sum_{V \in \Irr(\mathcal{C})} \nu_n(V) \dim_{\mathbb{C}}(V).
\end{equation*}
By equation~(\ref{eq:dimensions}), we have $\nu_n(H) = \nu_n(\Rep(H))$.

We first prove that the assignment $\mathcal{C} \mapsto \nu_n(\mathcal{C})$ is an invariant of fusion categories admitting a pivotal structure.

\begin{theorem}
  \label{thm:FS-invariance}
  Let $\mathcal{C}$ and $\mathcal{D}$ be pivotal fusion categories. If $\mathcal{C}$ and $\mathcal{D}$ are monoidally equivalent, we have $\nu_n(\mathcal{C}) = \nu_n(\mathcal{D})$ for every $n$.
\end{theorem}
\begin{proof}
  Let $F: \mathcal{C} \to \mathcal{D}$ be a monoidal equivalence. By Lemma~\ref{lem:FS-invariance}, $\nu_n(\mathcal{C})$ and $\nu_n(\mathcal{D})$ do not depend on the choice of pivotal structures. Hence we may assume that $F$ preserves the pivotal structures by changing pivotal structure of $\mathcal{D}$. Then, for every $V \in \Irr(\mathcal{C})$, we have
  \begin{equation*}
    \nu_n(F(V)) = \nu_n(V) \quad \text{and} \quad
    \pdim(F(V)) = \pdim(V)
  \end{equation*}
  by \cite[Corollary~4.4]{MR2381536} and~(\ref{eq:ptrace-1}). Now the proof is obvious.
\end{proof}

Let $\mathcal{C}$ and $\mathcal{D}$ be pivotal fusion categories. Then the exterior product $\mathcal{C} \boxtimes \mathcal{D}$ is naturally a pivotal fusion category. Recall that the objects of $\mathcal{C} \boxtimes \mathcal{D}$ are finite sums of the form $\bigoplus_i V_i \boxtimes W_i$ ($V_i \in \mathcal{C}$, $W_i \in \mathcal{D}$) and
\begin{equation*}
  \Hom_{\mathcal{C} \boxtimes \mathcal{D}}^{}(V \boxtimes W, V' \boxtimes W')
  = \Hom_{\mathcal{C}}^{}(V, V') \otimes \Hom_{\mathcal{D}}^{}(W, W').
\end{equation*}
One can easily see that $E_{V \boxtimes W}^{(n)} = E_V^{(n)} \otimes E_W^{(n)}$. We have $\nu_n(V \boxtimes W) = \nu_n(V) \, \nu_n(W)$ as remarked in the proof of Proposition~5.11 of \cite{MR2313527}. We also have $\pdim(V \boxtimes W) = \pdim(V) \pdim(W)$.

\begin{proposition}
  \label{prop:FS-product}
  $\nu_n(\mathcal{C} \boxtimes \mathcal{D}) = \nu_n(\mathcal{C})\,\nu_n(\mathcal{D})$.
\end{proposition}
\begin{proof}
  $(V, W) \mapsto V \boxtimes W$ gives a bijection between $\Irr(\mathcal{C}) \times \Irr(\mathcal{D})$ and $\Irr(\mathcal{C} \boxtimes \mathcal{D})$. The proof follows directly from the above observation.
\end{proof}

One may expect that $\nu_n(\mathcal{C})$ is an algebraic integer. In fact, the following proposition follows from Lemma~\ref{lem:cyclo-pdim} and Lemma~\ref{lem:cyclo-FS}.

\begin{proposition}
  \label{prop:FS-cyclo}
  There exists a root of unity $\xi$ such that $\nu_n(\mathcal{C}) \in \mathbb{Z}[\xi]$.
\end{proposition}

If $\mathcal{C}$ is $\Rep(H)$ for some semisimple Hopf algebra $H$, $\nu_2(\mathcal{C})$ is equal to the trace of the antipode $S$ \cite[Proposition~2.5]{MR2213320}. Since $S^2 = \id_H$ by the theorem of Larson and Radford, we have $\Trace(S) \in \mathbb{Z}$. The following proposition is motivated by this fact.

\begin{proposition}
  \label{prop:FS-reality-1}
  $\nu_2(\mathcal{C}) \in \mathbb{R}$.
\end{proposition}
\begin{proof}
  Let $V$ be a simple object of $\mathcal{C}$. Since $(E_V^{(2)})^2$ is identity and since
  \begin{equation*}
    \Hom_{\mathcal{C}}(\mathbf{1}, V^{\otimes 2}) \cong \Hom_{\mathcal{C}}(V^*, V) \cong
    \begin{cases}
      \mathbb{C} & \text{if $V$ is self-dual}, \\
      0          & \text{otherwise},
    \end{cases}
  \end{equation*}
  we have that $\nu_2(V) \ne 0$ if and only if $V$ is self-dual and that $\nu_2(V) \in \{ \pm 1 \}$ if $V$ is self-dual. Let $\mathcal{I}$ be the subset of $\Irr(\mathcal{C})$ consisting self-dual objects. Then we have
  \begin{equation*}
    \nu_2(\mathcal{C}) = \sum_{V \in \mathcal{I}} \nu_2(V) \pdim(V).
  \end{equation*}
  If $V$ is self-dual, $\pdim(V) = \pdim(V^*) = \overline{\pdim(V)}$ (for the last equality, see \cite[Proposition~2.9]{MR2183279}). Therefore, the right-hand side is a real number.
\end{proof}

One might expect moreover that $\nu_2(\mathcal{C})$ is an integer. However, this does not hold in general. In fact, if $\mathcal{C}$ is the Tambara-Yamagami category \cite[Definition~3.1]{MR1659954} associated with a finite abelian group $A$, a non-degenerate symmetric bicharacter $\chi$ of $A$ and a square root $\tau$ of $|A|^{-1}$, we have
\begin{equation*}
  \nu_2(\mathcal{C}) = \# \{ a \in A \mid a^2 = 1 \} + \sgn(\tau) \sqrt{|A|}
\end{equation*}
where $\sgn$ means the sign of a real number. This is not an integer unless $|A|$ is a square number. Indicators of Tambara-Yamagami categories will be discussed in a forthcoming paper.

Of course, it follows from the above proof, $\nu_2(\mathcal{C}) \in \mathbb{Z}$ if $\pdim(V) \in \mathbb{Z}$ for every self-dual simple object $V$. In particular, $\nu_2(H)$ for a semisimple quasi-Hopf algebra $H$ is an integer.

Higher indicators are not real in general. We show that $\nu_n(\mathcal{C}) \in \mathbb{R}$ for every $n$ if $\mathcal{C}$ admits a braiding. The following lemma is needed to prove this. Recall that also $\mathcal{C}^{\rm rev}$ is a pivotal fusion category.

\begin{lemma}
  \label{lem:FS-rev}
  $\nu_n(\mathcal{C}^{\rm rev}) = \overline{\nu_n(\mathcal{C})}$.
\end{lemma}
\begin{proof}
  By \cite[Lemma~5.2]{MR2381536}, we have $\nu_n(V^{\rm rev}) = \overline{\nu_n(V)}$. We also have
  \begin{equation*}
    \pdim(V^{\rm rev}) = \pdim(V^*) = \overline{\pdim(V)}
  \end{equation*}
  by (\ref{eq:ptrace-2}) \cite[Proposition~2.9]{MR2183279}. Now the proof is obvious.
\end{proof}

\begin{corollary}
  \label{cor:FS-reality-2}
  $\nu_n(\mathcal{C}) \in \mathbb{R}$ if $\mathcal{C}$ admits a braiding.
\end{corollary}
\begin{proof}
  If $\mathcal{C}$ admits a braiding, $\mathcal{C}$ is monoidally equivalent to $\mathcal{C}^{\rm rev}$. By Theorem~\ref{thm:FS-invariance} and Lemma~\ref{lem:FS-rev}, we have $\nu_n(\mathcal{C}) = \nu_n(\mathcal{C}^{\rm rev}) = \overline{\nu_n(\mathcal{C})}$.
\end{proof}

As another corollary of Lemma~\ref{lem:FS-rev}, we have the following description of indicators of the opposite category $\mathcal{C}^{\rm op}$.

\begin{corollary}
  \label{cor:FS-op}
  $\nu_n(\mathcal{C}^{\rm op}) = \overline{\nu_n(\mathcal{C})}$.
\end{corollary}
\begin{proof}
  The duality gives a monoidal equivalence between $\mathcal{C}$ and $\mathcal{C}^{\rm op\,rev}$. This implies that $\nu_n(\mathcal{C}) = \overline{\nu_n(\mathcal{C}^{\rm op})}$. Hence, $\nu_n(\mathcal{C}^{\rm op}) = \overline{\nu_n(\mathcal{C})}$.
\end{proof}

Now we assume $\mathcal{C}$ to be spherical. The following equation (\ref{eq:FS-formula}) is due to Ng and Schauenburg (see the proof of Theorem~5.5 of \cite{MR2313527}). We present the proof for the sake of completeness.

\begin{lemma}
  \label{lem:FS-formula}
  Let $\mathcal{C}$ be a spherical fusion category. We have
  \begin{equation}
    \label{eq:FS-formula}
    \nu_n(\mathcal{C}) = \frac{1}{\dim(\mathcal{C})}
    \sum_{X \in \Irr(\mathcal{Z}(\mathcal{C}))} \theta_X^n \pdim(X)^2
  \end{equation}
  where $\theta$ is the canonical twist of $\mathcal{Z}(\mathcal{C})$.
\end{lemma}
\begin{proof}
  We use formula~(\ref{eq:FS-third}). In view of~(\ref{eq:adjoint}), we have
  \begin{equation*}
    \begin{split}
      \nu_n(\mathcal{C})
      & = \frac{1}{\dim(\mathcal{C})} \sum_{V \in \Irr(\mathcal{C})} \ptrace(\theta_{I_{\mathcal{C}}(V)}^n) \pdim(V) \\
      & = \frac{1}{\dim(\mathcal{C})} \sum_{X \in \Irr(\mathcal{Z}(\mathcal{C}))}
      \sum_{V \in \Irr(\mathcal{C})} \theta_X^n \pdim(X) \cdot [\Pi_{\mathcal{C}}(X):V] \pdim(V).
    \end{split}
  \end{equation*}
  Here, since $\Pi_{\mathcal{C}}(X) \cong \bigoplus_{V \in \Irr(\mathcal{C})} V^{\oplus [\Pi_{\mathcal{C}}(X):V]}$, we have
  \begin{equation*}
    \pdim(X) = \pdim(\Pi_{\mathcal{C}}(X)) = \sum_{V \in \Irr(\mathcal{C})} [\Pi_{\mathcal{C}}(X):V] \pdim(V)
  \end{equation*}
  by the additivity of the pivotal dimension. This completes the proof.
\end{proof}

\begin{remark}
  {\rm (i)} Some authors defined the regular representation in $\mathcal{C}$ by using Frobenius-Perron dimensions instead of pivotal dimensions. One can prove
  \begin{equation*}
    \sum_{V \in \mathcal{C}} \nu_n(V) \FPdim(V)
    = \frac{1}{\dim(\mathcal{C})} \sum_{X \in \mathcal{Z}(\mathcal{C})} \theta_X^n \pdim(X) \FPdim(X).
  \end{equation*}
  in a similar way as the proof of Lemma~\ref{lem:FS-formula}.

  {\rm (ii)} The right hand side of equation~(\ref{eq:FS-formula}) is equal to the Reshetikhin-Turaev invariant of the lens space $L(n, 1)$ associated with modular tensor category $\mathcal{Z}(\mathcal{C})$, see \cite{MR1797619} and \cite{MR1292673}.
\end{remark}

The following corollaries are direct consequences of Lemma~\ref{lem:FS-formula}.

\begin{corollary}
  \label{cor:FS-center-1}
  Let $\mathcal{C}$ and $\mathcal{D}$ be spherical fusion categories. If $\mathcal{Z}(\mathcal{C})$ and $\mathcal{Z}(\mathcal{D})$ are equivalent as ribbon categories, we have $\nu_n(\mathcal{C}) = \nu_n(\mathcal{D})$ for every $n$.
\end{corollary}

\begin{corollary}
  \label{cor:FS-center-2}
  Let $\mathcal{C}$ and $\mathcal{D}$ be pseudo-unitary fusion categories. If $\mathcal{Z}(\mathcal{C})$ and $\mathcal{Z}(\mathcal{D})$ are equivalent as braided monoidal categories, we have $\nu_n(\mathcal{C}) = \nu_n(\mathcal{D})$ for every $n$.
\end{corollary}

Modular tensor categories \cite{MR1797619} are an important class of fusion categories. A formula of indicators of modular tensor categories has been obtained by Kashina, Montgomery and Ng: In the proof of Proposition~5.5 of \cite{KMN09}, they showed that
\begin{equation*}
  \nu_n(\mathcal{M}) = \frac{1}{\dim(\mathcal{M})} \left| \sum_{X \in \Irr(\mathcal{M})} \theta_X^n \pdim(X)^2 \right|^2
\end{equation*}
for a modular tensor category $\mathcal{M}$ with twist $\theta$. If $\mathcal{C}$ is a spherical fusion category, then $\mathcal{Z}(\mathcal{C})$ is a modular tensor category \cite{MR1966525}. Therefore:

\begin{theorem}[cf. {\cite[Theorem~5.6]{KMN09}}]
  \label{thm:FS-center-3}
  $\nu_n(\mathcal{Z}(\mathcal{C})) = |\nu_n(\mathcal{C})|^2$.
\end{theorem}

In the case where $\mathcal{C}$ is pseudo-unitary, we can give another proof: M\"uger showed that there is an equivalence $\mathcal{Z}(\mathcal{Z}(\mathcal{C})) \approx \mathcal{Z}(\mathcal{C} \boxtimes \mathcal{C}^{\rm rev})$ of braided monoidal categories in \cite{MR1966525}. Therefore Theorem~\ref{thm:FS-center-3} follows from Corollary~\ref{cor:FS-center-2}. It should be remarked that this argument is not applicable for general spherical fusion categories since we did not prove that $\mathcal{Z}(\mathcal{Z}(\mathcal{C}))$ and $\mathcal{Z}(\mathcal{C} \boxtimes \mathcal{C}^{\rm rev})$ are equivalent as ribbon categories.

We apply our results to Hopf algebras and obtain the following:

\begin{theorem}
  \label{thm:FS-Hopf}
  Let $H$ and $L$ be semisimple quasi-Hopf algebras. Denote by $D(H)$ the Drinfeld double of a quasi-Hopf algebra $H$.
  \begin{enumalph}
  \item Suppose that $\Rep(D(H))$ and $\Rep(D(L))$ are equivalent as braided monoidal categories. Then $\nu_n(H) = \nu_n(L)$ for every $n$.
  \item $\nu_n(H \otimes L) = \nu_n(H) \, \nu_n(L)$.
  \item $\nu_n(H^{\op\,\cop}) = \nu_n(H)$ and $\nu_n(H^{\op}) = \nu_n(H^{\cop}) = \overline{\nu_n(H)}$.
  \item $\nu_n(H) \in \mathbb{R}$ if $H$ admits a universal $R$-matrix.
  \item $\nu_n(D(H)) = |\nu_n(H)|^2$.
  \item If $H$ is a Hopf algebra, then $\nu_n(H^*)$, the $n$-th Frobenius-Schur indicator of the regular representation of the dual Hopf algebra $H^*$, is equal to $\nu_n(H)$.
  \end{enumalph}
\end{theorem}

The part (a), (b), (c), (d) and (e) are obvious. The part (f) follows from (a) and the well-known fact that $\Rep(D(H))$ and $\Rep(D(H^*))$ are equivalent as braided monoidal categories. Of course, this can be derived directly from the definition of $\nu_n(H)$. We note that the part (e) has appeared in \cite{KMN09}.

$H$ and $L$ are said to be {\em monoidally Morita equivalent} if $\Rep(H)$ and $\Rep(L)$ are equivalent as monoidal categories. By part (a) of the above theorem, we have that if $H$ and $L$ are monoidally Morita equivalent, then $\nu_n(H) = \nu_n(L)$ for every $n$, as we referred in Section~\ref{sec:introduction}.

\section{Group-theoretical categories}
\label{sec:group-theoretical}

\subsection{Formulae for group-theoretical categories}

In this section, we study indicators of group-theoretical categories. We briefly recall the definition. Let $\Gamma$ be a finite group and $\omega \in Z^3(\Gamma, \mathbb{C}^\times)$ a normalized 3-cocycle. If $F$ is a subgroup of $\Gamma$ and $\alpha: F \times F \to \mathbb{C}^\times$ is a function satisfying $\delta\alpha = \omega|_{F \times F \times F}$, $\mathbb{C}_\alpha F = \bigoplus_{x \in F} E_x$ is an algebra in $\Vect_\omega^\Gamma$ with multiplication given by $\alpha$. The category $\mathcal{C}(\Gamma, \omega, F, \alpha)$ is defined to be the monoidal category of $\mathbb{C}_\alpha F$-bimodules in $\Vect_\omega^\Gamma$. A fusion category $\mathcal{C}$ is said to be {\em group-theoretical} if it is monoidally equivalent to the category of such a form. Since every group-theoretical category is pseudo-unitary, it has a canonical pivotal structure \cite[Corollary~8.43]{MR2183279}.

\begin{theorem}
  \label{thm:FS-GT-1}
  For a group-theoretical category $\mathcal{C} = \mathcal{C}(\Gamma, \omega, F, \alpha)$, we have
  \begin{equation}
    \label{eq:FS-group-theoretical}
    \nu_n(\mathcal{C}) = \sum_{g \in \Gamma} \delta_{g^n, 1} \prod_{k = 1}^{n - 1} \omega(g, g^k, g).
  \end{equation}
\end{theorem}
\begin{proof}
  By \cite{MR2016657}, see also \cite{MR1822847}, there exists an equivalence $\mathcal{Z}(\mathcal{C}) \approx \mathcal{Z}(\Vect_\omega^\Gamma)$ of braided monoidal categories. Therefore, by Corollary~\ref{cor:FS-center-2}, we have
  \begin{equation*}
    \nu_n(\mathcal{C}) = \nu_n(\Vect_\omega^\Gamma) = \sum_{g \in \Gamma} \nu_n(E_g).
  \end{equation*}
  By Corollary~\ref{cor:FS-invertible-2}, we have the result.
\end{proof}

\begin{remark}
  \label{rem:fs-ind-gt}
  {\rm (i)} Note that the right-hand side of~(\ref{eq:FS-group-theoretical}) first appeared in the paper of Altsch\"uler and Coste \cite[(3.13)]{MR1230426} as the Dijkgraaf-Witten invariant of the lens space $L(n, 1)$ associated with $\Gamma$ and $\omega$.

  {\rm (ii)} For $\omega \in Z^3(\Gamma, \mathbb{C}^\times)$ and $n \ge 1$, define $\widetilde{\omega}_n: \Gamma \to \mathbb{C}$ by
  \begin{equation}
    \label{eq:omega-n}
    \widetilde{\omega}_{n}(g) = \delta_{g^n, 1} \prod_{k = 1}^{n - 1} \omega(g, g^k, g) \quad (g \in \Gamma).
  \end{equation}
  One can easily see that $\widetilde{\omega}_n(g)$ depends only on the cohomology class of the restriction of $\omega$ to the cyclic subgroup generated by $g$. In particular, if $\omega$ is cohomologous to the trivial $3$-cocycle, $\widetilde{\omega}_{n}(g) = \delta_{g^n, 1}$.

  {\rm (iii)} Altsch\"uler and Coste showed that $\widetilde{\omega}_n$ is a class function on $\Gamma$ (see \cite[Appendix]{MR1230426}). As the referee kindly pointed out, this follows from (ii) and the well-known fact in group cohomology that the function
  \begin{equation*}
    \omega^x(a, b, c) = (x a x^{-1}, x b x^{-1}, x c x^{-1}) \quad (a, b, c \in \Gamma)
  \end{equation*}
  is a 3-cocycle cohomologous to $\omega$.

  From the viewpoint of the theory of Frobenius-Schur indicators, the fact that $\widetilde{\omega}_n$ is a class function can be understood as follows: Let $x \in \Gamma$ and define an endofunctor $(-)^x$ on $\Vect_\omega^\Gamma$ by $V^x = (E_x \otimes V) \otimes E_x^*$ ($V \in \Vect_\omega^\Gamma$). Since $E_x$ is an invertible object, this functor extends to a monoidal autofunctor on $\Vect_\omega^\Gamma$. By the invariance of Frobenius-Schur indicators, we have $\nu_n(E_g^x) = \nu_n(E_g)$ for every $g \in \Gamma$. This implies that $\widetilde{\omega}_{n}$ is a class function on $\Gamma$.
\end{remark}

\subsection{Abelian extensions of group algebras}

Let $(F, G)$ be a matched pair of finite groups \cite{MR611561} together with maps
\begin{equation*}
  \triangleleft:  G \times F \to G \quad \text{and} \quad
  \triangleright: G \times F \to F.
\end{equation*}
The bicrossed product is denoted by $F \bicross G$; It is a set $F \times G$ endowed with multiplication given by $(x, g) \cdot (y, h) = (x(g \triangleright y), (g \triangleleft y)h)$ ($x, y \in F$, $g, h \in G$). The set of equivalence classes of abelian extensions
\begin{equation}
  \label{eq:extension}
  1 \to \mathbb{C}^G \to A \to \mathbb{C}F \to 1
\end{equation}
associated with this matched pair is denoted by $\Opext(\mathbb{C}F, \mathbb{C}^G)$. It is known that every Hopf algebra $A$ fitting into an abelian extension~(\ref{eq:extension}) can be constructed from maps $\sigma: G \times F \times F \to \mathbb{C}^\times$ and $\tau: G \times G \times F \to \mathbb{C}^\times$ satisfying certain cocycle conditions. The corresponding Hopf algebra, denoted by $A = \mathbb{C}^G \#_{\sigma, \tau} \mathbb{C}F$, is constructed as follows: For an algebra $K$, we mean by a {\em $K$-ring} an algebra equipped with an algebra map from $K$. As a $\mathbb{C}^G$-ring, $A$ is generated by $u_x$ ($x \in F$) with relations
\begin{equation}
  \label{eq:extension-mul}
  u_x u_y = \sum_{g \in G} \sigma(g; x, y) e_g u_{x y}
  \quad \text{and} \quad u_x e_g = e_{g \triangleleft x^{-1}} u_x
\end{equation}
where $e_g \in \mathbb{C}^G$ is the dual basis of $g \in G$. Note that $A$ is generated by $u_x$ ($x \in F$) and $e_g$ ($g \in G$) as an algebra. The comultiplication of $A$ is the algebra map $\Delta: A \to A \otimes A$ determined by
\begin{equation}
  \label{eq:extension-com}
  \Delta(u_x) = \sum_{h \in G} \tau(g, h; x) e_{g} u_{h \triangleright x} \otimes e_h u_x
  \text{\quad and \quad} \Delta(e_g) = \sum_{h \in G} e_{g h^{-1}} \otimes e_{h}.
\end{equation}
We omit a description of the antipode. For more details, see, e.g., \cite{MR1913439}. If $\sigma$ and $\tau$ are trivial, the corresponding Hopf algebra is called the {\em bismash product} and denoted by $\mathbb{C}^G \# \mathbb{C}F$.

Recall that we used the argument of Natale \cite{MR2016657} in the proof of Theorem~\ref{thm:FS-GT-1}. She also showed that if $A$ is a Hopf algebra fitting into~(\ref{eq:extension}), then $\Rep(A)$ is monoidally equivalent to $\mathcal{C}(F \bicross G, \omega, F, 1)$ where $\omega \in Z^3(F \bicross G, \mathbb{C}^\times)$ is the image of $A \in {\rm Opext}(\mathbb{C}^G, \mathbb{C}F)$ under the map $\overline{\omega}$ appearing in the Kac exact sequence
\begin{equation}
  \label{eq:Kac}
  \begin{split}
    \cdots &
    \mathop{\longrightarrow} H^2(F \bicross G, \mathbb{C}^\times)
    \mathop{\longrightarrow} H^2(F, \mathbb{C}^\times) \oplus H^2(G, \mathbb{C}^\times)
    \mathop{\longrightarrow}^{\partial} \Opext(\mathbb{C}^G, \mathbb{C}F) \\
    & \mathop{\longrightarrow}^{\overline{\omega}} H^3(F \bicross G, \mathbb{C}^\times)
    \mathop{\longrightarrow} H^3(F, \mathbb{C}^\times) \oplus H^3(G, \mathbb{C}^\times)
    \longrightarrow \cdots.
  \end{split}
\end{equation}
If $A = \mathbb{C}^G \#_{\sigma, \tau} \mathbb{C}F$, the corresponding $\omega$ is given by
\begin{equation}
  \label{eq:extension-3-cocycle}
  \omega(x, y, z)
  = \sigma(x_{|G}^{}; y_{|F}^{}, y_{|G}^{} \triangleright z_{|F}^{})
  \tau(x_{|G}^{} \triangleleft y_{|F}^{}, y_{|G}^{}; z_{|F}^{})
\end{equation}
where ${}_{|F}: F \bicross G \to F$ and ${}_{|G}: F \bicross G \to G$ are canonical projections. The following is a direct consequence of Theorem~\ref{thm:FS-GT-1}.

\begin{corollary}
  \label{cor:FS-extension-1}
  Notations are as above. For every $n$, we have
  \begin{equation*}
    \nu_n(\mathbb{C}^G \#_{\sigma, \tau} \mathbb{C}F) = \sum_{g \in F \bicross G} \delta_{g^n, 1} \prod_{k = 1}^{n - 1} \omega(g, g^k, g)
  \end{equation*}
  where $\omega$ is given by~(\ref{eq:extension-3-cocycle}).
\end{corollary}

More precisely, $\nu_n(\mathbb{C}^G \#_{\sigma, \tau} \mathbb{C}F)$ is given as follows: For a pair $(x, g) \in F \times G$, define $x_n(x, g) \in F$ and $g_n(x, g) \in G$ ($n \ge 1$) inductively by
\begin{align*}
  x_1(x, g) = x, \quad & x_{n + 1}(x, g) = x \cdot (g \triangleright x_n(x, g)), \\
  g_1(x, g) = g, \quad & g_{n + 1}(x, g) = (g_n(x, g) \triangleleft x) \cdot g
\end{align*}
so that $(x, g)^n = (x_n(x, g), g_n(x, g))$ in $F \bicross G$. Then
\begin{equation*}
  \nu_n(\mathbb{C}^G \#_{\sigma, \tau} \mathbb{C}F)
  = \sum_{x \in F, g \in G} \delta_{x_n, 1} \delta_{g_n, 1}
  \prod_{k = 1}^{n - 1} \sigma(g; x_k, g_k \triangleright x) \tau(g \triangleleft x_k, g_k; x),
\end{equation*}
where we abbreviated $x_k(x, g)$ and $g_k(x, g)$ as $x_k$ and $g_k$, respectively. Note that for $n = 2$ this formula has been showed by Kashina, Mason and Montgomery in \cite{MR1919158}.

If $A$ is isomorphic to the bismash product $\mathbb{C}^G \# \mathbb{C}F$, then the corresponding $\omega$ is trivial, and hence we have the following:

\begin{corollary}
  \label{cor:FS-extension-2}
  $\nu_n(\mathbb{C}^G \# \mathbb{C}F) = \# \{ x \in F \bicross G \mid x^n = 1 \}$
\end{corollary}

More precisely, the right hand side of the above equation is equal to the number of pairs $(x, g) \in F \times G$ such that $x_n(x, g) = 1$ and $g_n(x, g) = 1$. This has been showed by Kashina, Montgomery and Ng in \cite{KMN09} in a different way.

\subsection{Arithmetic properties of indicators}

We have shown that indicators of a pivotal fusion category $\mathcal{C}$ are cyclotomic integers in Proposition~\ref{prop:FS-cyclo}. This proposition can be refined in the case where $\mathcal{C}$ is group-theoretical. To start with, we observe the right-hand side of~(\ref{eq:FS-group-theoretical}) with $\Gamma = \mathbb{Z}_N$. For a while, we fix a positive integer $n$ and a normalized 3-cocycle $\omega \in Z^3(\mathbb{Z}_N, \mathbb{C}^\times)$. Let $\widetilde{\omega}_n: \mathbb{Z}_N \to \mathbb{C}$ denote the function given by~(\ref{eq:omega-n}) with $\Gamma = \mathbb{Z}_N$.

\begin{lemma}
  \label{lem:omega-n}
  Let $e$ be the cohomological order of $\omega$, that is, the order of the cohomology class of $\omega$ in $H^3(\mathbb{Z}_N, \mathbb{C}^\times)$. Suppose that $i \in \mathbb{Z}_N$ satisfies $n i = 0$. Then the following hold:
  \begin{enumalph}
  \item $\widetilde{\omega}_n(i)$ is a root of unity whose order divides all $N$, $n$ and $e$.
  \item $\widetilde{\omega}_n(a i) = \widetilde{\omega}_n(i)^{a^2}$ for every $a \in \mathbb{Z}_N$.
  \end{enumalph}
\end{lemma}
\begin{proof}
  Following a description of \cite[Appendix~E]{MR1002038}, $H^3(\mathbb{Z}_N, \mathbb{C}^\times)$ is a cyclic group of order $N$ generated by the cohomology class of
  \begin{equation}
    \label{eq:cohomology}
    \psi_N^{}(j, k, l) = \exp\left(
      \frac{2\pi\sqrt{-1}}{N^2} \cdot \overline{j\mathstrut}
      (\overline{k\mathstrut} + \overline{l\mathstrut} - \overline{k + l\mathstrut})
    \right) \quad (j, k, l \in \mathbb{Z}_N),
  \end{equation}
  where $\overline{a}$ denotes the integer between $0$ and $N - 1$ representing an element $a \in \mathbb{Z}_N$. Since the function $\widetilde{\omega}_n$ depends only on the cohomology class of $\omega$, we may assume that $\omega = (\psi_N^{})^r$ with $r = (N/e) \cdot s$ for some $s \in \mathbb{Z}$ coprime to $N$. Then, we compute
  \begin{equation*}
    \widetilde{\omega}_n(i)
    = \exp\left(
      \frac{2\pi r \sqrt{-1}}{N^2} \sum_{k = 1}^{n - 1} \overline{i\mathstrut}
      (\overline{i\mathstrut} + \overline{k i \mathstrut} - \overline{(k+1)i \mathstrut})
    \right)
    =  \exp\left(\frac{2\pi s \sqrt{-1}}{e} \cdot \frac{n (\overline{i})^2}{N} \right).
  \end{equation*}
  By an elementary number-theoretical argument, we can replace $\overline{i}$ in the above equation with $i$. Therefore, we obtain a formula
  \begin{equation*}
    \widetilde{\omega}_n(i) = \exp\left(\frac{2\pi s \sqrt{-1}}{e} \cdot \frac{n i^2}{N} \right)
    = \exp\left(\frac{2\pi s \sqrt{-1}}{n} \cdot \frac{n^2 i^2}{N e} \right).
  \end{equation*}
  Now (b) is obvious. Note that $e$ is a divisor of $N$. (a) follows from that both $n i^2 / N$ and $n^2 i^2 / e N$ are integers under our assumption that $n i = 0$ in $\mathbb{Z}_N$.
\end{proof}

\begin{lemma}
  \label{lem:GT-ind-cyclic-2}
  Notations are as above. Let $S = \sum_{i \in \mathbb{Z}_N} \widetilde{\omega}_n(i)$.
  \begin{enumalph}
  \item Suppose that $\omega$ is cohomologous to $(\psi_N^{})^r$ where $\psi_N^{} \in Z^3(\mathbb{Z}_N, \mathbb{C}^\times)$ is the normalized 3-cocycle given by~(\ref{eq:cohomology}). Then $S = S(n r / d, d)$ where $d = \gcd(N, n)$ is the greatest common divisor of $N$ and $n$ and
    \begin{equation*}
      S(a, m) := \sum_{i = 0}^{m - 1} \exp\left(\frac{2\pi\sqrt{-1}}{m} \cdot a i^2 \right)
      \quad (a \in \mathbb{Z}, m \in \mathbb{N}).
    \end{equation*}
  \item Let $e$ be the cohomological order of $\omega$. There exists an integer $a$ and a common divisor $c$ of $d$ and $e$ such that $S = (d/c) \cdot S(a, c)$.
  \end{enumalph}
\end{lemma}

The part (a) has been showed by Altsch\"uler and Coste \cite[(3.18)]{MR1230426}.

\begin{proof}
  The solutions $i$ of the congruence equation $n i \equiv 0 \pmod{N}$ are $i = (N / d) \cdot j$ ($j = 0, 1, \cdots, d - 1$). By Lemma~\ref{lem:omega-n}~(b), we have a formula
  \begin{equation}
    \label{eq:GT-ind-cyclic-2}
    S = \sum_{j = 0}^{d - 1} \widetilde{\omega}_n(N/d)^{j^2}.
  \end{equation}

  \textup{(a)} By arguments in the proof of Lemma~\ref{lem:omega-n}, $\widetilde{\omega}_n(N/d) = \exp(2\pi\sqrt{-1} \cdot n r/d^2)$. The proof is done by comparing the definition of $S(a, m)$ with~(\ref{eq:GT-ind-cyclic-2}).

  \textup{(b)} Let $c$ be the order of $\widetilde{\omega}_n(N/d)$. By Lemma~\ref{lem:omega-n}~(a), $c$ divides both $d$ and $e$. Fix $a \in \mathbb{Z}$ such that $\widetilde{\omega}_n(N/d) = \exp(2\pi\sqrt{-1} \cdot a/c)$. Then our claim follows immediately from~(\ref{eq:GT-ind-cyclic-2}).
\end{proof}

The above $S(a, m)$ is called the quadratic Gauss sum. It can be computed by the following Lemma~\ref{lem:quadratic-sum} which is well-known in the number theory. The reader may refer to, for example, \cite[Chapter~V]{MR0174513} for the proof.

\begin{lemma}
  \label{lem:quadratic-sum}
  Let $a$ and $m$ be positive integers. We denote by $\legendre{\ }{n}$ the Jacobi-Legendre symbol for a positive odd integer $n$.
  \begin{enumalph}
  \item $S(a + k m, m) = S(a, m)$ for every $k \in \mathbb{Z}$.
  \item For any common divisor $d$ of $a$ and $m$, $S(a, m) = d \cdot S(a / d, m / d)$.
  \item Suppose that $a$ and $m$ are relatively prime. Then, if $m$ is odd,
    \begin{equation*}
      S(a, m) = \legendre{a}{m} \sqrt{m} \times
      \begin{cases}
        1         & \text{if $m \equiv 1 \pmod{4}$}, \\
        \sqrt{-1} & \text{otherwise}.
      \end{cases}
    \end{equation*}
    If $m \equiv 0 \pmod{4}$,
    \begin{equation*}
      S(a, m) = \legendre{m}{a} \sqrt{m} \times
      \begin{cases}
        1 + \sqrt{-1} & \text{if $a \equiv 1 \pmod{4}$}, \\
        1 - \sqrt{-1} & \text{otherwise}.
      \end{cases}
    \end{equation*}
    If $m \equiv 2 \pmod{4}$, $S(a, m) = 0$.
  \end{enumalph}
\end{lemma}

Since every finite group is a union of its cyclic subgroups, above results are very helpful for our purpose. Fix a group-theoretical category $\mathcal{C} = \mathcal{C}(\Gamma, \omega, F, \alpha)$ till the end of this subsection. Let $c(\omega)$ be the least common multiple of the cohomological orders of the restrictions of $\omega$ to all cyclic subgroups of $\Gamma$. One can easily see that $c(\omega)$ divides both the exponent of $\Gamma$ and the cohomological order of $\omega$.

We give a refinement of Proposition~\ref{prop:FS-cyclo} by using $c(\omega)$. For a positive integer $m$, let $R(m)$ be the subring of $\mathbb{C}$ generated by the set $\{ S(a, d) \mid d|m, 0 < a < d \}$. For example, $R(1) = R(2) = \mathbb{Z}$, $R(3) = \mathbb{Z}[\sqrt{-3}]$ and $R(4) = \mathbb{Z}[2\sqrt{-1}]$. It is obvious by the definition that $R(m) \subset R(m')$ if $m$ divides $m'$.

\begin{lemma}
  \label{lem:GT-ind-arith-1}
  $\nu_n(\mathcal{C}) \in R(d)$ where $d = \gcd(n, c(\omega))$.
\end{lemma}
\begin{proof}
  Let $C_1, \cdots, C_m$ be all cyclic subgroups of $\Gamma$ whose orders are divisors of $n$. $C_{i_1} \cap \cdots \cap C_{i_k}$ is denoted by $C_{i_1 \cdots i_k}$ for convention. For a subset $X \subset \Gamma$, set
  \begin{equation}
    \label{eq:GT-partial-sum}
    W_n(X) = \sum_{g \in X} \delta_{g^n, 1} \prod_{k = 1}^{n - 1} \omega(g, g^k, g).
  \end{equation}
  By Theorem~\ref{thm:FS-GT-1}, $\nu_n(\mathcal{C}) = W_n(X)$ where $X = \{ g \in \Gamma \mid g^n = 1 \}$. Since $X = \bigcup_{i = 1}^m C_i$,
  \begin{equation*}
    \nu_n(\mathcal{C}) = \sum_{k = 1}^{m} \sum_{1 \le i_1 < \cdots < i_k \le m} (-1)^{k - 1} W_n(C_{i_1 \cdots i_k}).
  \end{equation*}
  Therefore, it suffices to prove that $W_n(C_{i_1 \cdots i_k}) \in R(d)$ for every $i_1, \cdots, i_k$. This follows from Remark~\ref{rem:fs-ind-gt}~(ii) and Lemma~\ref{lem:GT-ind-cyclic-2}~(b), since the cohomological order of the restriction of $\omega$ to $C_{i_1 \cdots i_k}$ divides both $n$ and $c(\omega)$.
\end{proof}

The following are direct consequences of the above lemma.

\begin{corollary}
  \label{cor:GT-ind-arith-2}
  $\nu_n(\mathcal{C}) \in R(c(\omega))$ for every $n \ge 1$.
\end{corollary}

\begin{corollary}
  Suppose that $c(\omega) \le 2$. Then $\nu_n(\mathcal{C}) \in \mathbb{Z}$ for every $n \ge 1$.
\end{corollary}

We are interested in arithmetic properties of indicators. Let $K$ be the subfield of $\mathbb{C}$ generated by $\{ \nu_n(\mathcal{C}) \}_{n \ge 1}$. $K$ has the following property:

\begin{proposition}
  $K/\mathbb{Q}$ is Galois. If $K \ne \mathbb{Q}$, the Galois group ${\rm Gal}(K/\mathbb{Q})$ is isomorphic to a direct product of finitely many $\mathbb{Z}_2$'s. In particular, the degree of the extension $K/\mathbb{Q}$ is a power of two.
\end{proposition}
\begin{proof}
  Let $p_1, \cdots, p_m$ be all prime divisors of $c(\omega)$. By Lemma~\ref{lem:quadratic-sum} and Corollary~\ref{cor:GT-ind-arith-2}, $K$ is a subfield $L = \mathbb{Q}[\sqrt{-1}, \sqrt{p_1}, \cdots, \sqrt{p_m}]$, which is the minimal splitting field of
  \begin{equation*}
    f(x) = (x^2 + 1) (x^2 - p_1) (x^2 - p_2) \cdots (x^2 - p_m) \in \mathbb{Q}[x].
  \end{equation*}
  $L/\mathbb{Q}$ is a Galois extension and the corresponding Galois group is isomorphic to $\mathbb{Z}_2^{m+1}$. Now the proof is obvious by the fundamental theorem of Galois theory.
\end{proof}

\begin{remark}
  Let $K$ be as above. Suppose that $\zeta \in K\setminus\{\pm 1\}$ is a root of unity. By the above proposition, ${\rm Gal}(\mathbb{Q}(\zeta)/\mathbb{Q})$ is isomorphic to a direct product of finitely many $\mathbb{Z}_2$'s. On the other hand, if the order of $\zeta$ is $m$, ${\rm Gal}(\mathbb{Q}(\zeta)/\mathbb{Q}) \cong \mathbb{Z}_m^{\times}$. It follows from these observations that $\zeta^{24} = 1$.

  This remark is motivated by the Frobenius-Schur indicators of the regular representations of semisimple Hopf algebras of dimension 27. We will use primitive third roots of unity to express them, see Theorem~\ref{thm:FS-ind-N3}.
\end{remark}

\subsection{Some estimations of $c(\omega)$}

In view of results of the last subsection, it is important to know $c(\omega)$ to study indicators of group-theoretical categories. Here we give the following estimation of $c(\omega)$.

\begin{lemma}
  \label{lem:GT-ind-estim}
  Let $\Gamma$ be a finite group and $\omega \in Z^3(\Gamma, \mathbb{C}^\times)$ a normalized 3-cocycle. Suppose that there exists a normal subgroup $H$ of $\Gamma$ such that the restriction of $\omega$ to $H$ is trivial. Then $c(\omega)$ divides the exponent of the quotient group $\Gamma/H$.
\end{lemma}
\begin{proof}
  Let $e$ be the exponent of $\Gamma/H$. By the definition, $e$ is the smallest positive integer such that $x^e \in H$ for every $x \in \Gamma$. Let $\omega_x$ denote the restriction of $\omega$ to the cyclic subgroup of $\Gamma$ generated by $x$. By the definition of $c(\omega)$, it is sufficient to show that $\omega_x^e$ is trivial for every $x \in \Gamma$.

  Let $x \in \Gamma$. Consider the restriction map $\rho_x: H^3(\langle x \rangle, \mathbb{C}^\times) \to H^3(\langle x^e \rangle, \mathbb{C}^\times)$. This is known to be surjective (see \cite[\S 6.7]{MR1269324}; Note that $H^*(G, \mathbb{C}^\times) \cong H^{*+1}(G, \mathbb{Z})$ for all group $G$). By the assumption that the restriction of $\omega$ to $H$ is trivial, $\omega_x \in \Ker(\rho_x)$. Hence we have
  \begin{equation*}
    |\Ker(\rho_x)|
    = \frac{|H^3(\langle x \rangle, \mathbb{C}^\times)|}{|{\rm Im}(\rho_x)|}
    = \frac{|H^3(\langle x \rangle, \mathbb{C}^\times)|}{|H^3(\langle x^e \rangle, \mathbb{C}^\times)|}
    = \frac{|\langle x \rangle|}{|\langle x^e \rangle|} = \gcd(m, e)
  \end{equation*}
  where $m$ is the order of $x$. This implies that $\omega_x^e$ is trivial.
\end{proof}

By using this estimation of $c(\omega)$, we give some criteria for indicators of group-theoretical categories being integers.

\begin{corollary}
  \label{cor:GT-int-criterion-1}
  Let $\mathcal{C} = \mathcal{C}(\Gamma, \omega, F, \alpha)$ be a group-theoretical category. If there exists a subgroup $H$ of $\Gamma$ of index two such that the restriction of $\omega$ to $H$ is trivial, then $\nu_n(\mathcal{C}) \in \mathbb{Z}$ for every $n$.
\end{corollary}
\begin{proof}
  As $H$ has index two, $H$ is normal in $\Gamma$. Note that $R(2) = \mathbb{Z}$. The result follows from Corollary~\ref{cor:GT-ind-arith-2} and Lemma~\ref{lem:GT-ind-estim}.
\end{proof}

\begin{corollary}
  \label{cor:GT-int-criterion-2}
  Suppose that $A$ is a semisimple Hopf algebra fitting into an abelian extension $1 \to \mathbb{C}^G \to A \to \mathbb{C}\mathbb{Z}_2 \to 1$. Then $\nu_n(A) \in \mathbb{Z}$ for every $n$.
\end{corollary}
\begin{proof}
  Let $\Gamma = G \bicross \mathbb{Z}_2$ be the bicrossed product associated with this extension. By the Kac exact sequence~(\ref{eq:Kac}), the restriction of $\omega = \overline{\omega}(A) \in Z^3(\Gamma, \mathbb{C}^\times)$ to $G$ is trivial. Recall that $\Rep(A)$ is equivalent to $\mathcal{C}(\Gamma, \omega, G, 1)$ as a monoidal category. By Corollary~\ref{cor:GT-int-criterion-1}, $\nu_n(A) \in \mathbb{Z}$ for every $n$.
\end{proof}

\section{Examples}
\label{sec:GT-example}

In this section, we compute Frobenius-Schur indicators of regular representations of some group-theoretical Hopf algebras.

Throughout this section, we use the following notations: For an element $g$ of a group, $\ord(g)$ means the order of $g$. $\mu_N$ is the group of $N$-th roots of unity in $\mathbb{C}$. For a positive integer $n$ and a prime number $p$, let $b_p(n)$ denote the maximal non-negative integer $i$ such that $p^i$ divides $n$. The Iverson bracket $[P]$ for a condition $P$ is used to indicate 1 if $P$ holds and 0 otherwise. For example,
\begin{equation*}
  [b_2(n) = 0] = [\text{$n$ odd}] = \begin{cases}
    1 & \text{if $n$ is odd}, \\
    0 & \text{if $n$ is even}.
  \end{cases}
\end{equation*}

\subsection{Some Hopf algebras of dimension $2N^2$}

Fix an integer $N > 1$. The right action of $F = \mathbb{Z}_2$ on $G = \mathbb{Z}_N \times \mathbb{Z}_N$ given by
\begin{equation*}
  (i, j) \triangleleft 0 = (i, j), \quad
  (i, j) \triangleleft 1 = (j, i)  \quad (i, j \in \mathbb{Z}_N)
\end{equation*}
and the trivial action of $G$ on $F$ make $(F, G)$ into a matched pair. Masuoka \cite[\S 2]{MR1448809} showed that $\Opext(\mathbb{C}F, \mathbb{C}^G) \cong \mu_N$. The Hopf algebra, denoted by $H_{2N^2}(\xi)$, obtained by the abelian extension corresponding to $\xi \in \mu_N$ is isomorphic to $\mathbb{C}^G \#_{\sigma, \tau} \mathbb{C}F$ with $\sigma$ and $\tau$ given by
\begin{equation*}
  \sigma((i, j); a, b) =
  \begin{cases}
    \xi^{i j} & \text{if $a = b = 1$}, \\
    1        & \text{otherwise},
  \end{cases}
  \quad \tau((i, j), (k, l); a) =
  \begin{cases}
    \xi^{j k} & \text{if $a = 1$}, \\
    1        & \text{otherwise}
  \end{cases}
\end{equation*}
for $i, j, k, l \in \mathbb{Z}_N$ and $a, b \in \mathbb{Z}_2$, see the proof of \cite[Theorem~2.1]{MR1448809}.

Applying Corollary~\ref{cor:FS-extension-1}, we compute Frobenius-Schur indicators of the regular representation of $H_{2N^2}(\xi)$. Let $\Gamma = F \bicross G$. By~(\ref{eq:extension-3-cocycle}), $\omega \in Z^3(\Gamma, \mathbb{C}^\times)$ associated to $H_{2N^2}(\xi)$ is given by
\begin{equation*}
  \omega((a_1, i_1, j_1), (a_2, i_2, j_2), (a_3, i_3, j_3)) =
  \begin{cases}
    1                       & \text{if $a_3 =   0$}, \\
    \xi^{j_1 i_2}           & \text{if $a_3 \ne 0$ and $a_2 =   0$}, \\
    \xi^{i_1 j_1 + i_1 i_2} & \text{if $a_3 \ne 0$ and $a_2 \ne 0$}
  \end{cases}
\end{equation*}
for $(a_k, i_k, j_k) \in \Gamma$ ($k = 1, 2, 3$).

\begin{theorem}
  \label{thm:GT-example}
  The $n$-th Frobenius-Schur indicator of the regular representation of $H = H_{2N^2}(\xi)$ is given as follows: If either $n$ is odd or the condition
  \begin{equation}
    \label{eq:GT-ex-1-cond}
    b_2(N) = b_2(\ord(\xi)) = b_2(n) - 1 \ge 1
  \end{equation}
  holds, then $\nu_n(H) = \gcd(N, n)^2$. Otherwise, $\nu_n(H) = \gcd(N, n)^2 + N \gcd(N, n/2)$.
\end{theorem}
\begin{proof}
  Let $\widetilde{\omega}_n: \Gamma \to \mathbb{C}$ be the function given by (\ref{eq:omega-n}). We first suppose that $n$ is odd. Then, $g \in \Gamma$ satisfies $g^n = 1$ if and only if $g = (0, i, j)$ with $i, j \in \mathbb{Z}_N$ satisfying $n i = n j = 0$. For such an element $g \in \Gamma$, we have $\widetilde{\omega}_n(g) = 1$. Hence, we obtain
  \begin{equation*}
    \nu_n(H) = \left( \# \{ i \in \mathbb{Z}_N \mid n i = 0 \} \right)^2 = \gcd(N, n)^2.
  \end{equation*}

  Next, we suppose that $n = 2 k$ is even. Then $g \in \Gamma$ satisfies $g^n = 1$ if and only if one of the following exclusive conditions holds:
  \begin{itemize}
  \item $g = (0, i, j)$ with $i, j \in \mathbb{Z}_N$ satisfying $n i = n j = 0$.
  \item $g = (1, i, j)$ with $i, j \in \mathbb{Z}_N$ satisfying $k(i + j) = 0$.
  \end{itemize}
  Let $g = (1, i, j) \in \Gamma$ be an element satisfying $g^n = 1$. Since $k (i + j) = 0$ in $\mathbb{Z}_N$, we may assume that $j = -i + (N/d) \cdot m$ ($0 \le m < d$) where $d = \gcd(N, k)$. We compute
  \begin{equation*}
    \widetilde{\omega}_n(g) = \prod_{l = 0}^{k - 1} \omega(g, g^{2l}, g) \omega(g, g^{2l+1}, g) = \xi^{E(k; i, j)}
  \end{equation*}
  where $E(k; i, j) \in \mathbb{Z}_N$, the exponent of $\xi$, is given by
  \begin{equation*}
    \sum_{l = 0}^{k - 1} \{ l j (i + j) + i j + i (l i + l j + i) \} = \binom{k}{2} \cdot (i + j)^2 = \binom{k}{2} \cdot \frac{N^2}{d^2} \cdot m^2.
  \end{equation*}
  Set $E(k) = \binom{k}{2} \cdot (N/d)^2$ so that $E(k; i, j) = E(k) m^2$. Since $2E(k) \equiv 0 \pmod{N}$, $\xi^{E(k)} = \pm 1$. We also see that $\xi^{E(k)} = -1$ only if both $N$ and $k$ are even. By the above arguments, we compute
  \begin{align*}
    \nu_n(H) & = \gcd(N, n)^2 + \sum_{i = 0}^{N - 1} \sum_{m = 0}^{d - 1} \xi^{E(k) m^2} \\
    & = \gcd(N, n)^2 + [\xi^{E(k)} = +1] \cdot N \gcd(N, k).
  \end{align*}
  Now it suffices to show that the condition~(\ref{eq:GT-ex-1-cond}) is equivalent to that $\xi^{E(k)} = -1$. This can be done by easy number-theoretical arguments.
\end{proof}

Suppose that there exists $\zeta \in \mu_N$ satisfying $\xi = \zeta^2$. (Such a $\zeta$ always exists when $N$ is odd.) Then, it follows from the above theorem that $\nu_n(H_{2N^2}(\xi))$ is equal to the number of elements $g \in \Gamma$ satisfying $g^n = 1$.

On the other hand, if there does not exist such a $\zeta \in \mu_N$, then $N$ is necessarily even, and we have $\nu_{2N}(H_{2N^2}(\xi)) = N^2 \ne \nu_{2N}(H_{2N^2}(1)) = 2N^2$. This implies that $H_{2N^2}(\xi)$ and $H_{2N^2}(1)$ are not monoidally Morita equivalent. Letting $N = 2$, we obtain the following result of Tambara and Yamagami \cite{MR1659954}. This result was also obtained by Ng and Schauenburg by using Frobenius-Schur indicators in a way slightly different from ours \cite[\S 6]{MR2366965}.

\begin{corollary}
  Let $\mathcal{B}_8 = H_8(-1)$, $D_8$ the dihedral group of order 8, $Q_8$ the quaternion group of order 8. Then $\mathcal{B}_8$, $\mathbb{C}D_8$ and $\mathbb{C}Q_8$ are not mutually monoidally Morita equivalent.
\end{corollary}
\begin{proof}
  Note that $H_8(1) \cong \mathbb{C}D_8$. $\mathcal{B}_8$ and $\mathbb{C}D_8$ are not monoidally Morita equivalent by the above observation. We also have $\nu_2(\mathbb{C}Q_8) = 2$ while $\nu_2(\mathcal{B}_8) = \nu_2(\mathbb{C}D_8) = 6$. This implies that both $\mathcal{B}_8$ and $\mathbb{C}D_8$ are not monoidally Morita equivalent to $\mathbb{C}Q_8$.
\end{proof}

\subsection{Some Hopf algebras of dimension $N^3$}

Fix an odd integer $N$. The right action of $F = \mathbb{Z}_N$ on $G = \mathbb{Z}_N \times \mathbb{Z}_N$ given by
\begin{equation*}
  (i, j) \triangleleft a = (i + a j, j) \quad (a, i, j \in \mathbb{Z}_N)
\end{equation*}
and the trivial action of $G$ on $F$ makes $(F, G)$ into a matched pair. Masuoka \cite[\S 3]{MR1448809} showed that $\Opext(\mathbb{C}^G, \mathbb{C}F) \cong \mu_N \times \mu_N$. We denote by $H_{N^3}(\xi, \zeta)$ the Hopf algebra obtained by the abelian extension corresponding to $(\xi, \zeta) \in \mu_N \times \mu_N$. Following the proof of \cite[Theorem~3.1]{MR1448809}, $H = H_{N^3}(\xi, \zeta)$ is constructed as follows: As a $\mathbb{C}^G$-ring, $H$ is generated by $x$ with relations
\begin{equation*}
  x^N = \sum_{i, j \in \mathbb{Z}_N} \xi^j e_{i j}
  \text{\quad and \quad} x e_{i j} = e_{i - j, j} x
\end{equation*}
where $e_{i j} \in \mathbb{C}^G$ is the dual basis of $(i, j) \in G$. $H$ is generated by $x$ and $e_{i j}$ ($i, j \in \mathbb{Z}_N$) as an algebra. The comultiplication of $H$ is the algebra map $\Delta$ determined by
\begin{equation*}
  \Delta(x) = \sum_{p, q, r, s \in \mathbb{Z}_N} \zeta^{q r} e_{p, q} x \otimes e_{r, s} x
  \quad \text{and} \quad \Delta(e_{i j}) = \sum_{p, q \in \mathbb{Z}_N} e_{i - p, j - q} \otimes e_{p q},
\end{equation*}
and the counit is the algebra map $\varepsilon$ determined by $\varepsilon(x) = 1$ and $\varepsilon(e_{i j}) = \delta_{i 0} \delta_{j 0}$.

Let us find cocycles $\sigma$ and $\tau$ corresponding to $H_{N^3}(\xi, \zeta)$. Fix an $N$-th root $\lambda$ of $\xi^{-1}$. For $j \in \mathbb{Z}_N$, set $\lambda_j = \lambda^m$ where $0 \le m < N$ is a unique integer such that $j \equiv m \pmod{N}$. Then, since the element
\begin{equation*}
  u_1 = \sum_{i, j \in \mathbb{Z}_N} \lambda_j e_{i j} x \in H_{N^3}(\xi, \zeta)
\end{equation*}
satisfies $u_1^N = 1$,  $u_a := u_1^a$ ($a \in \mathbb{Z}_N$) is well-defined. We have
\begin{equation*}
  \Delta(u_a) = \sum_{i, j, k, l \in \mathbb{Z}_N} \tau((i, j), (k, l); a) e_{i j} u_a \otimes e_{k l} u_a
\end{equation*}
where $\tau: G \times G \times F \to \mathbb{C}^\times$ is given by
\begin{equation*}
  \tau((i, j), (k, l); a)
  = (\lambda_{j + l} \lambda_j^{-1} \lambda_l^{-1})^a \cdot \zeta^{a \cdot j k + \binom{a}{2} \cdot j l}.
\end{equation*}
Note that the right-hand side of the above equation is well-defined for $a \in \mathbb{Z}_N$. Comparing these equations with (\ref{eq:extension-mul}) and (\ref{eq:extension-com}), we have that $H_{N^3}(\xi, \zeta)$ is isomorphic to $\mathbb{C}^G \#_{1, \tau} \mathbb{C}F$ with the above $\tau$.

Let $\Gamma = F \bicross G$. By~(\ref{eq:extension-3-cocycle}), the corresponding $\omega \in Z^3(\Gamma, \mathbb{C}^\times)$ is given by
\begin{equation*}
  \omega((a_1, i_1, j_1), (a_2, i_2, j_3), (a_3, i_3, j_3))
  = (\lambda_{j_1 + j_2}^{} \lambda_{j_1}^{-1} \lambda_{j_2}^{-1})^{a_3} \zeta^{a_3 \cdot j_1 i_2 + \binom{a_3}{2} j_1 j_2}
\end{equation*}
where $a_k, i_k, j_k \in \mathbb{Z}_N$ ($k = 1, 2, 3$). Let $\widetilde{\omega}_n: \Gamma \to \mathbb{C}$ be the function given by~(\ref{eq:omega-n}). To compute indicators, we provide the following lemma.

\begin{lemma}
  Let $n$ be a positive integer.
  \begin{enumalph}
  \item $g = (a, i, j) \in \Gamma$ satisfies $g^n = 1$ if and only if $n a = n i = n j = 0$ in $\mathbb{Z}_N$.
  \item Let $g = (a, i, j) \in \Gamma$ be an element satisfying $g^n = 1$. Then
    \begin{equation*}
      \widetilde{\omega}_n(g) = \xi^{a j n / N} \cdot \zeta^{\binom{n}{3} a^2 j^2}.
    \end{equation*}
  \end{enumalph}
\end{lemma}
\begin{proof}
  (a) We have $g^n = (n a, n i + \binom{n}{2} a j, n j)$ by induction on $n$. This implies that $g^n = 1$ if and only if  $n a = n i + \binom{n}{2} a j = n j = 0$ in $\mathbb{Z}_N$. Since $N$ is odd, $\binom{n}{2} a j = 0$ under the assumption that $n a = 0$. Therefore, this is equivalent to that $n a = n i = n j = 0$.

  (b) Suppose that $0 \le a, j < N$. By the definition of $\omega$, we have
  \begin{align*}
    \widetilde{\omega}_n(g)
    & = \prod_{k = 1}^{n - 1} \tau\left(i, j, k i + \binom{k}{2} a j, k j; a\right) \\
    & = \prod_{k = 1}^{n - 1} (\lambda_{j + k j} \lambda_j^{-1} \lambda_{k j}^{-1})^a
    \cdot \zeta^{a j (k i + \binom{k}{2} a j) + \binom{a}{2} k j^2}
    = \lambda_j^{- n a} \cdot \zeta^{E(n; a, i, j)}
  \end{align*}
  where $E(n; a, i, j) \in \mathbb{Z}_N$, the exponent of $\zeta$, is given by
  \begin{equation*}
    E(n; a, i, j) = a^2 j^2 \sum_{k = 1}^{n - 1} \binom{k}{2}
    + \left\{ a i + \binom{a}{2} j \right\} \cdot j \sum_{k = 1}^{n - 1} k
    = \binom{n}{3} a^2 j^2.
  \end{equation*}
  We also have $\lambda_j^{- n a} = (\lambda_j^N)^{- \frac{1}{N} n a} = \xi^{a j n / N}$. This completes the proof.
\end{proof}

\begin{theorem}
  \label{thm:FS-ind-N3}
  The $n$-th Frobenius-Schur indicator of the regular representation of $H = H_{N^3}(\xi, \zeta)$ is given as follows: Let $\alpha = \xi^{E_1(n)}$ and $\beta = \zeta^{E_2(n)}$ where
  \begin{equation*}
    E_1(n) = \frac{N n}{\gcd(N, n)^2} \text{\quad and \quad}
    E_2(n) = \binom{n}{3} \frac{N^4}{\gcd(N, n)^4}.
  \end{equation*}
  If both the condition
  \begin{equation}
    \label{eq:GT-ex-2-cond} b_3(n) = b_3(N) = b_3(\ord(\zeta)) \ge 1
  \end{equation}
  and the inequality $b_3(\ord(\xi)) \le 1$ hold, then
  \begin{equation*}
    \nu_n(H) = \frac{\gcd(N, n)^3}{9 \ord(\alpha)} \times \begin{cases}
      \ \ 5 + 4 \beta & \text{if $b_3(\ord(\xi)) = 0$}, \\
      3(5 - 2 \beta)  & \text{if $b_3(\ord(\xi)) = 1$}.
    \end{cases}
  \end{equation*}
  Otherwise, $\nu_n(H) = \gcd(N, n)^3 / \ord(\alpha)$.
\end{theorem}
\begin{proof}
  Let $d = \gcd(N, n)$. Note that the solutions $x$ of the congruence equation $n x \equiv 0 \pmod{N}$ are $x = (N/d) \cdot r$ ($0 \le r < d$). By the previous lemma,
  \begin{equation*}
    \nu_n(H) = d \sum_{r, s = 0}^{d - 1} \alpha^{r s} \beta^{r^2 s^2}.
  \end{equation*}
  Since $3 E_2(n) \equiv 0 \pmod{N}$, $\beta$ is a third root of unity. It follows from easy number-theoretical arguments that $\beta \ne 1$ if and only if the condition~(\ref{eq:GT-ex-2-cond}) holds.

  Suppose that $\beta = 1$. Then, since $\alpha^d = 1$,
  \begin{equation*}
    \nu_n(H) = d \sum_{r, s = 0}^{d - 1} (\alpha^r)^s
    = d^2 \cdot \# \{ 0 \le r < d \mid \alpha^r = 1 \}
    = \frac{d^3}{\ord(\alpha)}.
  \end{equation*}

  Suppose that $\beta \ne 1$. Then $d$ is divisible by $3$. Let $m = \ord(\alpha^3)$. We compute
  \begin{align*}
    \nu_n(H) & = d \sum_{s = 0}^{\mathstrut d - 1} \sum_{r = 0}^{\mathstrut d/3 - 1}
    (\alpha^{3 r s} + \alpha^{(3r + 1)s} \beta^{s^2} + \alpha^{(3r + 2)s}) \beta^{s^2}. \\
    & = d \sum_{s = 0}^{\mathstrut d - 1} (1 + \alpha^s \beta^{s^2} + \alpha^{2s} \beta^{s^2}) \sum_{r = 0}^{d/3 - 1} (\alpha^{3s})^r \\
    & = \frac{d^2}{3} \sum_{i = 0}^{d/m - 1} (1 + \alpha^{m i} \beta^{(m i)^2} + \alpha^{2 m i} \beta^{(m i)^2}).
  \end{align*}
  Note that $d/m$ is divisible by 3. Since $\alpha^{3m} = \beta^3 = 1$, the $i$-th summand of the above sum is equal to $3$ if $i \equiv 0 \pmod{3}$ and $1 + (\alpha^m + \alpha^{2m}) \beta^{m^2}$ otherwise. Therefore
  \begin{equation*}
    \nu_n(H) = \frac{d^3}{9m}(5 + 2(\alpha^m + \alpha^{2m})\beta^{m^2}).
  \end{equation*}
  Note that $b_3(\ord(\xi)) = b_3(\ord(\alpha))$ since we have assumed~(\ref{eq:GT-ex-2-cond}). The rest of the proof is a case-by-case analysis according to $e = b_3(\ord(\xi))$. We have
  \begin{equation*}
    \alpha^m + \alpha^{2m} = \begin{cases}
      2  & (e = 0), \\
      -1 & (e = 1), \\
      -1 & (e \ge 2),
    \end{cases}
    \quad \beta^{m^2} = \begin{cases}
      \beta & (e = 0), \\
      \beta & (e = 1), \\
      1     & (e \ge 2),
    \end{cases}
    \quad \ord(\alpha) = \begin{cases}
      m  & (e = 0), \\
      3m & (e = 1), \\
      3m & (e \ge 2).
    \end{cases}
  \end{equation*}
  By compiling these equations, we complete the proof.
\end{proof}

This theorem shows that $\nu_n(H_{N^3}(\xi, \zeta))$ does not depends on $\zeta$ if $N$ is coprime to 3. The situation is different and very interesting when $N = 3$. Fix a primitive third root $\beta$ of unity. Then, by the classification result of Masuoka \cite[Theorem~3.1]{MR1364343}, non-commutative Hopf algebras of dimension $27$ are $H_{27}(\beta^i, \beta^j)$ ($i = 0, 1$; $j = 0, 1, 2$) up to isomorphism. By Theorem~\ref{thm:FS-ind-N3}, we have Table~\ref{tab:FS-ind-27} of Frobenius-Schur indicators of the regular representations of them.

\begin{table}
  \begin{tabular}{c|cccccc}
    $H$ & $\nu_1(H)$ & $\nu_3(H)$ & $\nu_9(H)$ & $\nu_{27}(H)$ \\ \hline
    $H_{27}(1,     1)$       & $1$ & $27$ & $27$ & $27$ \\
    $H_{27}(\beta, 1)$       & $1$ &  $9$ & $27$ & $27$ \\
    $H_{27}(1,     \beta)$   & $1$ & $3(5 + 4\beta^2)$ & $27$ & $27$ \\
    $H_{27}(\beta, \beta)$   & $1$ & $3(5 - 2\beta^2)$ & $27$ & $27$ \\
    $H_{27}(1,     \beta^2)$ & $1$ & $3(5 + 4\beta)$   & $27$ & $27$ \\
    $H_{27}(\beta, \beta^2)$ & $1$ & $3(5 - 2\beta)$   & $27$ & $27$
  \end{tabular}
  \caption{Indicators of regular representations of $H_{27}(\xi, \zeta)$}
  \label{tab:FS-ind-27}
\end{table}

This table gives us various information on categories of representations of them. For example, we have the following by Theorem~\ref{thm:FS-Hopf}:

\begin{corollary}
  Hopf algebras $H_{27}(\beta^i, \beta^j)$ $(i = 0, 1; j = 0, 1, 2)$ are not mutually monoidally Morita equivalent.
  $H_{27}(\beta^i, \beta^j)$ $(i = 0, 1$; $j = 1, 2)$ are neither quasitriangular nor coquasitriangular.
\end{corollary}

\subsection{A family of braided Hopf algebras of Suzuki}

\subsubsection{Suzuki's construction}

Suzuki \cite{MR1637640} introduced a family of finite-dimensional Hopf algebras $A_{N L}^{\alpha\beta}$, parametrized by integers $N \ge 1$, $L \ge 2$ and $\alpha, \beta \in \{ \pm 1 \}$, to give a systematic description of co\-semi\-simple Hopf algebras of low-dimensions. In fact, as Suzuki remarked, every non-trivial semisimple Hopf algebras of dimension eight and of dimension twelve belong to this family (in characteristic zero).

We recall his construction. As an algebra, $A_{N L}^{\alpha\beta}$ is generated by $x_{i j}$ ($i, j \in \mathbb{Z}_2$) with defining relations
\begin{gather*}
  x_{00}^2 = x_{11}^2, \quad
  x_{01}^2 = x_{10}^2, \quad
  x_{00}^{2N} + \alpha x_{01}^{2N} = 1, \quad
  x_{ij}^{} x_{kl}^{} = 0 \quad \text{(if $i - j \ne k - l$)}, \\
  \underbrace{x_{00}^{} x_{11}^{} x_{00}^{} \cdots}_L
  = \underbrace{x_{11}^{} x_{00}^{} x_{11}^{} \cdots}_L, \quad
  \underbrace{x_{01}^{} x_{10}^{} x_{01}^{} \cdots}_L
  = \beta \underbrace{x_{10}^{} x_{01}^{} x_{10}^{} \cdots}_L.
\end{gather*}
The coalgebra structure of $A_{N L}^{\alpha\beta}$ is given by
\begin{equation*}
  \Delta(x_{ij}) = x_{i 0} \otimes x_{0 j} + x_{i 1} \otimes x_{1 j}, \quad
  \varepsilon(x_{ij}) = \delta_{ij}
\end{equation*}
and the antipode is given by $S(x_{ij}) = x_{ji}^{4N-1}$. Suzuki showed that it has a set
\begin{equation*}
  \{ x_{00}^s \underbrace{x_{11}^{} x_{00}^{} x_{11}^{} \cdots}_t,
  \, x_{01}^s \underbrace{x_{10}^{} x_{01}^{} x_{10}^{} \cdots}_t
  \mid 1 \le s \le 2N, 0 \le t < L \}.
\end{equation*}
as a basis, and hence $\dim_{\mathbb{C}}(A_{N L}^{\alpha\beta}) = 4NL$.

Fix a quadruple $(N, L, \alpha, \beta)$ and denote $A_{NL}^{\alpha\beta}$ by $A$. Let $G^\vee$ be the group generated by central grouplike elements $h_{\pm} = x_{00}^2 \pm x_{01}^2 \in A$. $\mathbb{C}G^\vee$ is canonically isomorphic to $\mathbb{C}^G$ where $G = G^{\vee\vee}$ is the character group of $G^\vee$. As remarked in \cite[Remark~3.4]{MR1637640}, $A$ fits into an extension
\begin{equation}
  1 \longrightarrow \mathbb{C}^G \longrightarrow A \mathop{\longrightarrow}^{\pi} \mathbb{C}D_{2L} \longrightarrow 1.
\end{equation}
Here, $D_{2L}$ is the dihedral group of order $2L$ which has a Coxeter presentation
\begin{equation}
  D_{2L} = \left\langle s_0^{}, s_1^{} \mid s_0^2 = s_1^2 = (s_0 s_1)^L = 1 \right\rangle.
\end{equation}
The quotient map $\pi$ is given by $\pi(x_{ij}) = \delta_{ij} s_i$ ($i, j \in \mathbb{Z}/2\mathbb{Z}$).

Our first task is to find cocycles $\sigma$ and $\tau$ such that $A \cong \mathbb{C}^G \#_{\sigma, \tau} \mathbb{C}D_{2L}$. Following the theory of $\mathbb{C}^G$-rings, this is same as to find invertible elements $u_x \in A$ ($x \in D_{2L}$) such that $(\id_A \otimes \pi) \Delta(u_x) = u_x \otimes x$. For this purpose, we introduce some notations. Let $\widetilde{e}_i = \alpha^i x_{0i}^{2N}$ ($i \in \mathbb{Z}_2$). They are central orthogonal idempotents satisfying $\widetilde{e}_0^{} + \widetilde{e}_1^{} = 1$. For a non-negative integer $m$ and $i, j \in \mathbb{Z}_2$, we set
\begin{equation*}
  \chi_{ij}^{(m)} = \widetilde{e}_{i-j}^{} \cdot \underbrace{x_{ij}^{} x_{i+1, j+1}^{} x_{ij}^{} \cdots}_m.
\end{equation*}
The following lemma is useful for computation.

\begin{lemma}
  \label{lem:suzuki-eq}
  The following equations hold in $A$ for $0 \le m < 2L$:
  \begin{gather*}
    \Delta(\chi_{ij}^{(m)})
    = \chi_{i 0}^{(m)} \otimes \chi_{0 j}^{(m)} + \chi_{i 1}^{(m)} \otimes \chi_{1 j}^{(m)}, \\
    \chi_{11}^{(m)} = h_+^{-L+m} \chi_{00}^{(2L-m)}, \quad
    \chi_{01}^{(m)} = \beta h_+^{-L+m} \chi_{10}^{(2L-m)}.
  \end{gather*}
\end{lemma}

Every element $x \in D_{2L}$ can be expressed as
\begin{equation*}
  x = \underbrace{s_0 s_1 s_0 \cdots}_{l} \quad (0 \le l < 2L)
\end{equation*}
in a unique way. This number $l$ is denoted by $\ell(x)$. Remark that, in general, $\ell(x)$ is {\em not} the minimal length of $x$ with respect to generators $s_0$ and $s_1$. Set $r = s_0 s_1$ and $s = s_0$. Then $D_{2L}$ has another group presentation
\begin{equation*}
  D_{2L} = \langle r, s \mid r^L = s^2 = 1, s r s^{-1} = r^{-1} \rangle.
\end{equation*}
Every element $x \in D_{2L}$ can be expressed as $x = r^i s^j$ ($0 \le i < L; j = 0, 1$) in a unique way. By using these $i$ and $j$, we have $\ell(x) = 2i + j$.

Using notations introduced in above, we define $u_x \in A$ ($x \in D_{2L}$) by
\begin{equation*}
  u_x = h_+^{-\floor{l/2}} (\chi_{00}^{(l)} + \eta^l \chi_{10}^{(l)})
\end{equation*}
where $l = \ell(x)$, $\eta$ is a fixed $2L$-th root of $\beta$ and $\floor{\ }$ is the floor function. By Lemma~\ref{lem:suzuki-eq}, $(\id_A \otimes \pi)\Delta(u_x) = u_x \otimes x$ for all $x \in D_{2L}$. The following lemma shows that all $u_x$ are invertible.

\begin{lemma}
  \label{lem:suzuki-cocycle-1}
  $u_x u_y = \widetilde{\sigma}(x, y) u_{xy}$ for all $x, y \in D_{2L}$ where
  \begin{equation*}
    \widetilde{\sigma}(x, y)
    = \left\{ (\widetilde{e}_0 + \eta^{2 \ell(y)} \widetilde{e}_1)
      \cdot h_+^{[\text{\rm $\ell(y)$ odd}]}
    \right\}^{[\text{\rm $\ell(x)$ odd}]} \in (\mathbb{C}G^\vee)^\times.
  \end{equation*}
\end{lemma}

Here, Iverson bracket is used to make the definition of $\widetilde{\sigma}(x, y)$ compact.

\begin{proof}
  Note that, by definition,
  \begin{equation*}
    u_r = h_+^{-1} (\widetilde{e}_0 x_{00} x_{01} + \eta^2 \cdot \widetilde{e}_1 x_{10} x_{01})
    \text{\quad and \quad}
    u_s = \widetilde{e}_0 x_{00} + \eta \cdot \widetilde{e}_1 x_{10}.
  \end{equation*}
  Verify the following equalities by using Lemma~\ref{lem:suzuki-eq}:
  \begin{equation*}
    u_r^L = 1, \quad
    u_s \cdot u_r = (\widetilde{e}_0 + \eta^4 \widetilde{e}_1) u_{sr}, \quad
    u_s \cdot u_s = (\widetilde{e}_0 + \eta^2 \widetilde{e}_1) h_+.
  \end{equation*}
  Let $x \in D_{2L}$. Then $\ell(x)$ is even (resp. odd) if and only if $x = r^i$ (resp. $x = r^i s$) for some $i$. Note that $u_{r^i} = (u_r)^i$ and $u_{r^i s} = (u_r)^i \cdot u_s$. The proof can be done by a case-by-case analysis according to parities of $\ell(x)$ and $\ell(y)$.
\end{proof}

We define an automorphism $x \mapsto \overline{x}$ on $D_{2L}$ by $\overline{s_0^{}} = s_{1}^{}$ and $\overline{s_1^{}} = s_0^{}$.

\begin{lemma}
  \label{lem:suzuki-cocycle-2}
  $\Delta(u_x) = u_x \otimes \widetilde{e}_0 u_x + u_{\overline{x}} \otimes \widetilde{e}_1 u_x$ for all $x \in D_{2L}$.
\end{lemma}
\begin{proof}
  Since the claim is obvious when $x = 1$, we assume that $x \ne 1$. Then,
  \begin{equation*}
    \overline{x}
    = \underbrace{s_1^{} s_0^{} s_1^{} \cdots}_{\ell(x)}
    = \underbrace{s_0^{} s_1^{} \cdots s_0^{} s_1^{}}_{2L} \underbrace{s_1^{} s_0^{} s_1^{} \cdots}_{\ell(x)}
    = \underbrace{s_0^{} s_1^{} s_0^{} \cdots}_{2L-\ell(x)}.
  \end{equation*}
  Hence we have $\ell(\overline{x}) = 2L - \ell(x)$. By Lemma~\ref{lem:suzuki-eq},
  \begin{align*}
    \Delta(h_+^{\floor{l/2}} u_x)
    & = \chi_{00}^{(l)} \otimes \chi_{00}^{(l)} + \chi_{01}^{(l)} \otimes \chi_{10}^{(l)}
    + \eta^l (\chi_{10}^{(l)} \otimes \chi_{00}^{(l)} + \chi_{11}^{(l)} \otimes \chi_{10}^{(l)}) \\
    & = (\chi_{00}^{(l)} + \eta^l \chi_{10}^{(l)}) \otimes \chi_{00}^{(l)}
    + h_+^{- L + l} (\chi_{00}^{(l')} + \eta^{l'} \chi_{10}^{(l')}) \otimes \eta^l \chi_{10}^{(l)}
  \end{align*}
  where $l = \ell(x)$ and $l' = \ell(\overline{x})$. Note that $\floor{\frac{1}{2}l'} = L + \floor{\frac{1}{2}l} - l$. We obtain
  \begin{equation*}
    \Delta(u_x) = u_x \otimes \widetilde{e}_0 u_x + u_{\overline{x}} \otimes \widetilde{e}_1 u_x
  \end{equation*}
  by multiplying both sides by $\Delta(h_+^{-\floor{l/2}}) = h_+^{-\floor{l/2}} \otimes h_+^{-\floor{l/2}}$.
\end{proof}

\subsubsection{Cyclic case}

We assume that $(N, \alpha) \ne ({\rm even}, +1)$. Then, $G^\vee$ is a cyclic group of order $2N$ generated by $h = x_{00}^2 - \alpha x_{01}^2$ \cite[Proposition~3.3]{MR1637640}. The character group $G = G^{\vee\vee}$ of $G^\vee$ is a cyclic group of order $2N$ generated by a group homomorphism
\begin{equation*}
  b: G^\vee \to \mathbb{C}^\times, \quad h^k \mapsto \zeta_{2N}^k \quad (k \in \mathbb{Z})
\end{equation*}
where $\zeta_{2N}^{}$ is a fixed primitive $2N$-th root of unity. By the orthogonality relation of characters, the dual basis $\{ e_k \}$ of $\{ b^k \}$ is given by
\begin{equation*}
  e_k = \frac{1}{2N} \sum_{i = 0}^{2N} \zeta_{2N}^{-ik} h^i \quad (k = 0, 1, \cdots, 2N - 1).
\end{equation*}
The elements $\widetilde{e}_0, \widetilde{e}_1, h_\pm \in \mathbb{C}G^\vee$ can be written as linear combinations of $e_k$ as
\begin{equation*}
  \widetilde{e}_0^{} = \sum_{i = 0}^{N-1} e_{2i}, \quad
  \widetilde{e}_1^{} = \sum_{i = 0}^{N-1} e_{2i+1} \quad \text{and} \quad
  h_{\pm} = \sum_{i = 0}^{2N-1} (\mp\alpha\zeta_{2N})^i \, e_i.
\end{equation*}
Comparing Lemma~\ref{lem:suzuki-cocycle-1} and Lemma~\ref{lem:suzuki-cocycle-2} with (\ref{eq:extension-mul}) and (\ref{eq:extension-com}), we have the following.

\begin{lemma}
  Assume that $(N, \alpha) \ne (\text{even}, +1)$. The right group action of $G$ on $D_{2L}$ given by $b \triangleright x = \overline{x} \ (x \in D_{2L})$ and the trivial action of $D_{2L}$ on $G$ make $(G, D_{2L})$ into a matched pair. Define $\sigma: G \times D_{2L} \times D_{2L} \to \mathbb{C}^\times$ by
  \begin{equation*}
    \sigma(b^i; x, y) = \left\{ \eta^{2 \ell(y) [\text{\rm $i$ odd}]}
      \cdot (-\alpha \zeta_{2N})^{i [\text{\rm $\ell(y)$ odd}]}
    \right\}^{[\text{\rm $\ell(x)$ odd}]}.
  \end{equation*}
  $A_{NL}^{\alpha\beta}$ is isomorphic to $\mathbb{C}^G \#_{\sigma, 1} \mathbb{C}D_{2L}$ as a Hopf algebra.
\end{lemma}

Let $\Gamma_{NL} = G \bicross D_{2L}$ be the bicrossed product. Both $G$ and $D_{2L}$ are regarded as subgroups of $\Gamma_{NL}$. Note that $\Gamma_{NL}$ has a presentation
\begin{equation*}
  \Gamma_{NL} = \left\langle b, r, s
    \left|
      \begin{gathered}
        b^{2N} = r^L = s^2 = 1, s r s^{-1} = r^{-1}, \\
        b r b^{-1} = r^{-1}, b s b^{-1} = r^{-1} s
      \end{gathered}
    \right.
  \right\rangle.
\end{equation*}
By~(\ref{eq:extension-3-cocycle}), the 3-cocycle $\omega \in Z^3(\Gamma_{NL}, \mathbb{C}^\times)$ associated with $A$ is given by
\begin{equation*}
  \omega(b^i x, b^j y, b^k z) = \left\{
    \eta^{(-1)^j 2 \ell(z) [\text{\rm $i$ odd}]}
    \cdot (-\alpha\zeta_{2N})^{i [\text{\rm $\ell(z)$ odd}]}
  \right\}^{[\text{\rm $\ell(y)$ odd}]}
\end{equation*}
for $i, j, k \in \mathbb{Z}_{2N}$ and $x, y, z \in D_{2L}$.

\begin{remark}
  \label{rem:suzuki-c-1}
  $c(\omega) \le 2$. This follows from Corollary~\ref{lem:GT-ind-estim}, since $\langle b, r \rangle$ is a subgroup of index 2 such that the restriction of $\omega$ to this subgroup is trivial.
\end{remark}

\begin{theorem}
  \label{thm:suzuki-FS-ind-cyc}
  Suppose that $(N, \alpha) \ne ({\rm even}, +1)$. Then the $n$-th Frobenius-Schur indicator of the regular representation of $A = A_{NL}^{\alpha\beta}$ is given by
  \begin{align*}
    \nu_n(A) & = \gcd(N, n) \gcd(L, n) \\
    & + [b_2(n) - 1 \ge b_2(N)] \cdot 2L \gcd(N, n) \\
    & + [b_2(n) - 1 \ge b_2(N), b_2(L)] \cdot \varepsilon_{N L}(n) \gcd(N, n) \gcd(L, n)
  \end{align*}
  where
  \begin{equation*}
    \varepsilon_{N L}(n) = (-\alpha)^{[b_2(n) = 1]} (-1)^{[b_2(n) = b_2(N) + 1]} \beta^{[b_2(n) = b_2(L) + 1]}.
  \end{equation*}
\end{theorem}

\begin{proof}
  Let $\widetilde{\omega}_n: \Gamma \to \mathbb{C}$ denote the function given by~(\ref{eq:omega-n}). For positive integers $r$ and  $s$, let $r \sslash s := r / \gcd(r, s)$. Note that $r x \equiv 0 \pmod{s}$ if and only if $x \equiv 0 \pmod{r \sslash s}$. Equations like $[\text{$r \sslash s$ odd}] = [b_2(r) \le b_2(s)]$ will be used many times.

  Let $H$ be the subgroup of $\Gamma_{N L}$ generated $b^2$ and $r$. It is a normal subgroup of $\Gamma_{N L}$ isomorphic to $\mathbb{Z}_N \times \mathbb{Z}_L$. We have
  \begin{equation*}
    \nu_n(A) = W_n(H) + W_n(b H) + W_n(s H) + W_n(b s H)
  \end{equation*}
  where $W_n(X)$ ($X \subset \Gamma_{N L}$) is given by the same formula as~(\ref{eq:GT-partial-sum}).

  Suppose that $n$ is odd. Let $g \in \Gamma_{N L}$ with $g^n = 1$. Then $g \in H$, since every element of $\Gamma_{N L} \setminus H$ has an even order. Since the restriction of $\omega$ to $H$ is trivial, we have
  \begin{equation*}
    \nu_n(A) = \#\{ g \in H \mid g^n = 1 \} = \gcd(N, n) \gcd(L, n).
  \end{equation*}
  Hence the equation holds for $n$.

 Suppose that $n = 2k$ is even. Then, $W_n(H) = \gcd(N, n) \gcd(L, n)$ in the same way as above. We compute (i) $W_n(bH)$, (ii) $W_n(s H)$ and (iii) $W_n(b s H)$ in what follows.

 (i) Since $\omega|_{b H \times b H \times bH} \equiv 1$, we have $W_n(b H) = \# \{ g \in b H \mid g^n = 1 \}$. Let $g \in b H$. Then $g = b^{2i + 1} r^j$ for some $i, j$. $g^n = 1$ if and only if $k(2i + 1) \equiv 0 \pmod{N}$. If $N \sslash k$ is even, this congruence equation has no solutions. Otherwise, the number of this solutions is $\gcd(N, k)$. Summarizing, we have
 \begin{equation*}
   W_n(b H) = [\text{$N \sslash k$ odd}] \cdot L \gcd(N, k) = [b_2(n) \ge b_2(N) + 1] \cdot L \gcd(N, n).
 \end{equation*}

 (ii) Let $g \in s H$. Then $g = b^{2 i} r^j s$ for some $i, j$. We have that $\widetilde{\omega}_n(g) = \delta_{g^n, 1} \cdot \zeta_{2N}^{2 i k}$ and that $g^n = 1$ if and only if $2 i k \equiv 0 \pmod{N}$. If $N \sslash k$ is odd, then the solutions of this congruence equation are $i = (N \sslash k) \cdot m$ ($0 \le m < \gcd(N, k)$). For these $i$, we have $\zeta_{2N}^{2 i k} = 1$, and hence $W_n(s H) = L \cdot \gcd(N, k)$. On the other hand, if $N \sslash k$ is even, the solutions of this congruence equation are $i = \frac{1}{2} (N \sslash k) \cdot m$ ($0 \le  m < 2 \gcd(N, k)$). For these $i$, we have $\zeta_{2N}^{2 i k} = (-1)^m$, and hence $W_n(s H) = 0$. Summarizing, we have
 \begin{equation*}
   W_n(s H) = [b_2(n) \ge b_2(N) + 1] \cdot L \gcd(N, n).
 \end{equation*}

 (iii) Let $g \in b s H$. Then $g = b^{2i+1} r^j s$ for some $i, j$. We have that $\widetilde{\omega}_n(g) = \delta_{g^n, 1} \cdot (-\alpha \eta^{-2(2j+1)} \zeta_{2N}^{2i+1})^k$ and that $g^n = 1$ if and only if both the following congruence equations hold:
 \begin{align}
   \label{eq:suzuki-cong-1} (2i + 1) k & \equiv 0 \pmod{N} \\
   \label{eq:suzuki-cong-2} (2j + 1) k & \equiv 0 \pmod{L}
 \end{align}
 If $N \sslash k$ is even, the congruence equation (\ref{eq:suzuki-cong-1}) has no solutions. Similarly, (\ref{eq:suzuki-cong-2}) has no solutions if $L \sslash k$ is even. Therefore, $W_n(b s H) = 0$ if either $N \sslash k$ or $L \sslash k$ is even.

 Suppose that both $N \sslash k$ and  $L \sslash k$ are odd. Write $N \sslash k = 2 i_0 + 1$. Then the solutions of the congruence equation~(\ref{eq:suzuki-cong-1}) are $i = i_0 + (N \sslash k) \cdot m$ ($0 \le m < \gcd(N, k)$). For these $i$, $(2i+1)k \equiv k \sslash N \pmod{2N}$. Since $N \sslash k$ is odd,
 \begin{equation*}
   \zeta_{2N}^{k(2i+1)} = (-1)^{k \sslash N} = (-1)^{[b_2(k) = b_2(N)]} = (-1)^{[b_2(n) = b_2(N) + 1]}.
 \end{equation*}
 Similarly, if $j$ satisfies~(\ref{eq:suzuki-cong-2}), $\eta^{-2k(2j+1)} = \beta^{[b_2(n) = b_2(L) + 1]}$. Summarizing, we have
 \begin{align*}
   W_n(b s H) & = [b_2(n) - 1 \ge b_2(N), b_2(L)] \cdot \varepsilon_{N L}(n) \gcd(N, n) \gcd(L, n).
 \end{align*}

 Combining (i), (ii) and (iii), we have the result.
\end{proof}

\subsubsection{Non-cyclic case}

We next suppose that $(N, \alpha) = ({\rm even}, +1)$. Then $G^\vee$ has a presentation
\begin{equation*}
  G^\vee = \langle h_+, h_- \mid h_+ h_- = h_- h_+, h_{\pm}^N = 1, h_+^2 = h_-^2 \rangle.
\end{equation*}
Define group homomorphism $a, b: G^\vee \to \mathbb{C}^\times$ by
\begin{equation*}
  \langle a, h_{\pm} \rangle = \zeta_N, \quad \langle b, h_{\pm} \rangle = \pm 1
\end{equation*}
where $\zeta_N$ is a fixed primitive $N$-th root of unity. Let $(e_0^+, \cdots, e_{N - 1}^+, e_0^{-}, \cdots, e_{N - 1}^{-})$ be the dual basis of $(1, a, \cdots, a^{N - 1}, b, a b, \cdots, a b^{N - 1})$. We have
\begin{gather*}
  e_{k}^\pm = \frac{1}{2N} (1 \pm \zeta_N^{-k} h_-) \sum_{i = 0}^{N - 1} \zeta_N^{-k i} h_+^i
  \quad (k = 0, 1, \cdots, N - 1), \\
  \widetilde{e}_0 = \sum_{i = 0}^{N - 1} e_i^+, \quad
  \widetilde{e}_1 = \sum_{i = 0}^{N - 1} e_i^-, \quad \text{and} \quad
  h_{\pm} = \sum_{i = 0}^{N - 1} \zeta_N^i (e_i^+ \pm e_i^-).
\end{gather*}
Comparing Lemma~\ref{lem:suzuki-cocycle-1} and Lemma~\ref{lem:suzuki-cocycle-2} with~(\ref{eq:extension-mul}) and~(\ref{eq:extension-com}), we have the following:

\begin{lemma}
  Assume that $(N, \alpha) = (\text{\rm even}, +1)$. Then the left group action of $G$ on $D_{2L}$ given by $a \triangleright x = x$, $b \triangleright x = \overline{x}$ and the trivial action of $D_{2L}$ on $G$ make $(G, D_{2L})$ into a matched pair. Define a map $\sigma: G \times D_{2L} \times D_{2L} \to \mathbb{C}^\times$ by
  \begin{equation*}
    \sigma(a^i b^j; x, y) = \left\{
      \eta^{2 \ell(y) [\text{\rm $j$ odd}]}
      \cdot ((-1)^j \zeta_N^i)^{[\text{\rm $\ell(y)$ odd}]}
    \right\}^{[\text{\rm $\ell(x)$ odd}]}
  \end{equation*}
  Then $A_{NL}^{+\beta} \cong \mathbb{C}^G \#_{\sigma, 1} \mathbb{C}D_{2L}$ as a Hopf algebra.
\end{lemma}

Let $\Gamma'_{NL} = G \bicross D_{2L}$ and $\omega \in Z^3(\Gamma'_{N L}, \mathbb{C}^\times)$ the normalized 3-cocycle associated with $A_{N L}^{+\beta}$. $G$ and $D_{2L}$ are regarded as subgroups of $\Gamma'_{N L}$. $\Gamma'_{N L}$ is decomposed into a direct sum of $\langle a \rangle$ and $\langle b, r, s \rangle$. The former subgroup is a cyclic group of order $N$, and the latter is isomorphic to $\Gamma_{1 L}$.

\begin{remark}
  \label{rem:suzuki-c-2}
  $c(\omega) \le 2$. This follows from Corollary~\ref{lem:GT-ind-estim}, since $\langle a, b, r \rangle$ is a subgroup of index 2 such that the restriction of $\omega$ to this subgroup is trivial.
\end{remark}

\begin{theorem}
  \label{thm:suzuki-FS-ind-noncyc}
  The $n$-th Frobenius-Schur indicator of the regular representation of $A = A_{N L}^{+\beta}$ with $N$ even is given by
  \begin{align*}
    \nu_n(A) & = \gcd(N, n) \cdot \gcd(L, n) \\
    & + [b_2(n) -1 \ge 0] \cdot L \gcd(N, n) \\
    & + [b_2(n) -1 \ge b_2(N)] \cdot L \gcd(N, n) \\
    & + [b_2(n) -1 \ge b_2(N), b_2(L)] \cdot \varepsilon_{N L}'(n) \gcd(N, n) \gcd(L, n)
  \end{align*}
  where
  \begin{equation*}
    \varepsilon_{N L}'(n) = (-1)^{[b_2(n) = 1]} \beta^{[b_2(n) - 1 = b_2(L)]}.
  \end{equation*}
\end{theorem}

\begin{proof}
  The proof is very similar to that of Theorem~\ref{thm:suzuki-FS-ind-cyc}. $W_n(X)$ and $r \sslash s$ have the same meanings as in the proof of Theorem~\ref{thm:suzuki-FS-ind-cyc}. Let $H$ be the subgroup of $\Gamma'_{NL}$ generated by $a$ and $r$. Then we have
  \begin{equation*}
    \nu_n(A) = W_n(H) + W_n(b H) + W_n(s H) + W_n(b s H).
  \end{equation*}

  Suppose that $n$ is odd. Then in a similar way as the proof of Theorem~\ref{thm:suzuki-FS-ind-cyc},
  \begin{equation*}
    \nu_n(A) = \# \{ g \in H \mid g^n = 1 \} = \gcd(N, n) \gcd(L, n).
  \end{equation*}

  Suppose that $n = 2k$ is even. Then $W_n(H) = \gcd(N, n) \gcd(L, n)$. We compute (i) $W_n(b H)$, (ii) $W_n(s H)$ and (iii) $W_n(b s H)$ in what follows.

  (i) Since $\omega|_{b H \times b H \times bH} \equiv 1$, we have $W_n(b H) = \# \{ g \in b H \mid g^n = 1 \}$. Let $g \in b H$. Then $g = a^i b r^j $ for some $i, j$. $g^n = 1$ if and only if $2 i k \equiv 0 \pmod{N}$. The number of solutions of this congruence equation is $\gcd(N, 2k) = \gcd(N, n)$. Summarizing, we have $W_n(b H) = L \gcd(N, n)$.

  (ii) Let $g \in s H$. Then $g = a^i r^j s$ for some $i, j$. We have that $\widetilde{\omega}_n(g) = \delta_{g^n, 1} \cdot \zeta_N^{i k}$. In a similar way as (ii) of the proof of Theorem~\ref{thm:suzuki-FS-ind-cyc},
  \begin{equation*}
    W_n(s H) = [\text{$N \sslash k$ odd}] \cdot L \gcd(N, k) = [b_2(n) - 1 \ge b_2(N)] \cdot L \gcd(N, n).
  \end{equation*}

  (iii) Let $g \in b s H$. Then $g = a^i b r^j s$ for some $i, j$. We have that $\widetilde{\omega}_n(g) = \delta_{g^n, 1} \cdot (-\eta^{-2(2j+1)} \zeta_N^{i})^k$ and that $g^n = 1$ if and only if both of the following holds:
  \begin{align}
    \label{eq:suzuki-cong-3} 2 i k & \equiv 0 \pmod{N}, \\
    \label{eq:suzuki-cong-4} (2j + 1) k & \equiv 0 \pmod{L}.
  \end{align}

  The latter congruence equation has no solutions if $L \sslash k$ is even. Suppose that $L \sslash k = 2 j_0 + 1$ is odd. Then the solutions of~(\ref{eq:suzuki-cong-4}) are $j = j_0 + (L \sslash k) \cdot m$ ($0 \le m < \gcd(L, k)$). For these $j$, $2k(2j + 1) \equiv (k \sslash L) \cdot 2L \pmod{4L}$. Since $L \sslash k$ is odd,
  \begin{equation*}
    \eta^{-2k(2j + 1)} = \beta^{k \sslash L} = \beta^{[b_2(n) = b_2(L) + 1]}.
  \end{equation*}

  Now we compute $W_n(b H s)$ as follows:
  \begin{equation*}
    W_n(b H s) = [b_2(n) - 1 \ge b_2(L)] \cdot \varepsilon'_{N L}(n) \gcd(L, n) \sum_{i} \zeta_N^{i k}
  \end{equation*}
  where $i$ runs through the solutions of~(\ref{eq:suzuki-cong-3}). In a similar way as (ii) of the proof of Theorem~\ref{thm:suzuki-FS-ind-cyc}, we have that the sum is equal to $[b_2(n) - 1 \ge b_2(N)] \gcd(N, n)$. Therefore,
  \begin{equation*}
    W_n(b H s) = [b_2(n) - 1 \ge b_2(L), b_2(N)] \cdot \varepsilon'_{N L}(n) \gcd(L, n) \gcd(N, n).
  \end{equation*}

  Combining (i), (ii) and (iii), we have the result.
\end{proof}

\section{Frobenius theorem}
\label{sec:frobenius}

We conclude this paper by raising a question motivated by a theorem of Frobenius on the number of solutions of the equation $x^n = 1$ in finite groups. For a finite group $G$ and a positive integer $n$, let $G[n] = \{ g \in G \mid g^n = 1 \}$. In 1895, Frobenius proved the following theorem:

\begin{theorem}[Frobenius]
  \label{thm:Frobenius-org}
  If $n$ divides $|G|$, then $n$ divides $|G[n]|$.
\end{theorem}

There are numerous proofs of this theorem, see, e.g., \cite{MR1157226}. Now recall that $\nu_n(\mathbb{C}G) = |G[n]|$. By using Frobenius-Schur indicators, we formulate the theorem of Frobenius for pivotal fusion categories as follows:

\begin{definition}
  \label{def:frobenius}
  Let $\mathcal{C}$ be a pivotal fusion category such that $\dim(\mathcal{C})$ is a positive integer. We say that {\em Frobenius theorem holds for $\mathcal{C}$} if $\nu_n(\mathcal{C})$ is divisible by $n$ (in the ring of algebraic integers) for every divisor $n$ of $\dim(\mathcal{C})$.
\end{definition}

Note that $\dim \Rep(H) = \dim_{\mathbb{C}}(H)$ for a semisimple quasi-Hopf algebra $H$. We say that Frobenius theorem holds for $H$ if it holds for $\Rep(H)$. Theorem~\ref{thm:Frobenius-org} claims that, in our terms, ``Frobenius theorem holds for every finite group algebra''. Our question is the following:

\begin{question}
  Does Frobenius theorem hold for every semisimple Hopf algebra?
\end{question}

There exists a semisimple quasi-Hopf algebra for which Frobenius theorem does not hold: Let $\Gamma = \mathbb{Z}_p$ with $p$ an odd prime with $p \equiv 1 \pmod{4}$ and $\omega = \psi_p^{}$ the 3-cocycle given by~(\ref{eq:cohomology}) with $N = p$. Then Frobenius theorem does not hold for the quasi-Hopf algebra $H = \mathbb{C}^{\Gamma}_{\omega}$. In fact, since $\Rep(H) \approx \Vect_{\omega}^{\Gamma}$, we have $\nu_p(H) = \sqrt{p}$ by results of Section~\ref{sec:group-theoretical}. $p$ does not divides $\nu_p(H)$ while $p$ does $\dim_{\mathbb{C}}(H) = p$.

In the rest of this section, we give some partial results on Frobenius theorem for semisimple Hopf algebras, especially group-theoretical ones. The following is obvious by Theorem~\ref{thm:FS-Hopf} (f).

\begin{theorem}
  Frobenius theorem holds for a semisimple Hopf algebra $H$ if and only if it holds for the dual Hopf algebra $H^*$.
\end{theorem}

By results of Section~\ref{sec:group-theoretical} and Section~\ref{sec:GT-example}, we have the following:

\begin{theorem}
  \label{thm:Frobenius-hopf}
  Let $H$ be a semisimple Hopf algebra. Frobenius theorem holds for $H$ if one of the following holds:
  \begin{enumalph}
  \item $H$ is isomorphic to a bismash product $\mathbb{C}^G \# \mathbb{C}F$.
  \item $H$ is isomorphic to $H_{2N^2}(\xi)$.
  \item $H$ is isomorphic to $H_{N^3}(\xi, \zeta)$.
  \item $H$ is isomorphic to the Suzuki's Hopf algebra $A_{N L}^{\alpha\beta}$.
  \end{enumalph}
\end{theorem}

(a) is a direct consequence of the Frobenius theorem for finite groups (see Corollary~\ref{cor:FS-extension-2}). Even if we know explicit values of indicators, (b), (c) and (d) of the above theorem are cumbersome to check. In the rest of this section, we show a more general result that covers both (b) and (d). Let $\mathcal{C} = \mathcal{C}(\Gamma, \omega, F, \alpha)$ be a fixed group-theoretical category, till the end of this section.

\begin{theorem}
  \label{thm:Frobenius-1}
  Let $n$ be a divisor of $|\Gamma|$. Suppose that $p = \gcd(n, c(\omega))$ is a prime.
  \begin{enumalph}
  \item If $p = 2$, $\nu_n(\mathcal{C})$ is divisible by $n$.
  \item If $p$ is odd, $\nu_n(\mathcal{C})$ is divisible by $\frac{n}{\sqrt{p}}$.
  \end{enumalph}
\end{theorem}

We point out that Isaacs and Robinson gave a nice and elementary proof of Theorem~\ref{thm:Frobenius-org} in \cite{MR1157226}. Our proof is based on their idea.

\begin{proof}
  By Theorem~\ref{thm:FS-GT-1}, $\nu_n(\mathcal{C}) = \sum_{g \in \Gamma[n]} \widetilde{\omega}_n(g)$ where $\widetilde{\omega}_n: \Gamma \to \mathbb{C}$ is the function given by~(\ref{eq:omega-n}). Let $g \in \Gamma$. There exists a unique decomposition $g = x \cdot y = y \cdot x$ such that the order of $x$ is a power of $p$ and the order of $y$ is coprime to $p$. $x$ and $y$ are called the $p$-part and $p'$-part of $g$, respectively. Write $n = p^k \cdot q$ with $q$ coprime to $p$. Then the assignment $g \mapsto (x, y)$ gives a bijection between $\Gamma[n]$ and
  \begin{equation*}
    \Phi := \{ (x, y) \mid x \in \Gamma[p^k], y \in C_\Gamma(x), y^q = 1 \}
  \end{equation*}
  where $C_\Gamma(x)$ is the centralizer of $x$ in $\Gamma$.

  Recall the construction of the $p$-part of $g \in \Gamma[n]$. By elementary number theory, there exist $a, b \in \mathbb{Z}$ such that $a p^k + b q = 1$. Then the $p$-part of $g$ is given by $x = g^{b q}$. Note that $b q \equiv 1 \pmod{p^k}$. By Lemma~\ref{lem:omega-n},
  \begin{equation*}
    \widetilde{\omega}_n(x) = \widetilde{\omega}_n(g^{b q}) = \widetilde{\omega}_n(g)^{b^2 q^2} = \widetilde{\omega}_n(g).
  \end{equation*}
  This means that $\widetilde{\omega}_n(g)$ depends only on the $p$-part of $g$.

  We first show that $\nu_n(\mathcal{C})$ is divisible by $q$. Note that $\Gamma[p^k]$ is invariant under conjugation. Let $x_1, \cdots, x_r$ be representatives of conjugacy classes contained in $\Gamma[p^k]$. Since $\widetilde{\omega}_n$ is a class function,
  \begin{equation*}
    \nu_n(\mathcal{C}) = \sum_{i = 1}^{r} \frac{|\Gamma|}{|C_\Gamma(x_i)|}
    \cdot |C_\Gamma(x_i)[q]| \cdot \widetilde{\omega}_n(x_i).
  \end{equation*}
  By Theorem~\ref{thm:Frobenius-org}, $|C_\Gamma(x_i)[q]| = m_i \cdot \gcd(|C_\Gamma(x_i)|, q)$ for some positive integer $m_i$. Every summand of the right-hand side of the above equation is divisible by $q$, since
  \begin{equation*}
    \frac{|\Gamma|}{|C_\Gamma(x_i)|} \cdot |C_\Gamma(x_i)[q]|
    = \frac{|\Gamma| \cdot m_i \gcd(|C_\Gamma(x_i)|, q)}{|C_\Gamma(x_i)|}
    = \frac{|\Gamma| \cdot m_i}{\lcm(|C_\Gamma(x_i)|, q)} \cdot q
  \end{equation*}
  where $\lcm(a, b)$ is the least common multiple of $a$ and $b$. This implies that $\nu_n(\mathcal{C})$ is divisible by $q$.

  Next, we argue the divisibility of $\nu_n(\mathcal{C})$ by powers of $p$. Let
  \begin{equation*}
    X = \{ g \in \Gamma[p^k] \mid \widetilde{\omega}_n(g) \ne 1 \} \quad \text{and} \quad
    \Phi' = \{ (x, y) \in \Phi | x \in X \}.
  \end{equation*}
  Then, we have
  \begin{equation}
    \label{eq:thm-frob-1}
    \nu_n(\mathcal{C}) = |\Phi \setminus \Phi'| + \sum_{x \in X} |C_\Gamma(x)[q]| \cdot \widetilde{\omega}_n(x).
  \end{equation}

  Note that every element of $X$ has order $p^k$. Indeed, if $x \in \Gamma[p^{k - 1}]$, $\widetilde{\omega}_{n/p}(x)$ is a $p$-th root of unity by Lemma~\ref{lem:omega-n}, and hence $\widetilde{\omega}_n(x) = \widetilde{\omega}_{n/p}(x)^p = 1$. For every $x \in X$ and $m \in \mathbb{Z}$ coprime to $p$, $x^m \in X$ by Lemma~\ref{lem:omega-n}. It follows from these observations that $M := \mathbb{Z}_{p^k}^\times$ acts freely on $X$ by $m \rightharpoonup x = x^m$. Let $x_1, \cdots, x_s$ be complete representatives of orbits of this action. Note that, if $x$ and $x'$ belong to the same orbit, $C_\Gamma(x) = C_\Gamma(x')$. Hence,
  \begin{equation*}
    |\Phi'| = \sum_{i = 1}^s |M| \cdot |C_\Gamma(x_i)[q]| = \sum_{i = 1}^s p^{k - 1} (p - 1) |C_\Gamma(x_i)[q]|.
  \end{equation*}
  By Lemma~\ref{lem:omega-n}, the second term of the right-hand side of~(\ref{eq:thm-frob-1}) is computed as follows:
  \begin{align*}
    \sum_{x \in X} |C_\Gamma(x)[q]| \cdot \widetilde{\omega}_n(x)
    & = \sum_{i = 1}^{s} \sum_{m \in M} |C_\Gamma(x_i)[q]| \cdot \widetilde{\omega}_n(x_i)^{m^2} \\
    & = \sum_{i = 1}^{s} \sum_{m = 1}^{p - 1} p^{k - 1} |C_\Gamma(x_i)[q]| \cdot \widetilde{\omega}_n(x_i)^{m^2}.
  \end{align*}
  Let $S_i = \sum_{m = 0}^{p - 1} \widetilde{\omega}_n(x_i)^{m^2}$. Combining the above equations, we have
  \begin{equation*}
    \nu_n(\mathcal{C}) = |\Gamma[n]|
    - p^k \sum_{i = 1}^{s} |C_\Gamma(x_i) [q]|
    + p^{k - 1} \sum_{i = 1}^{s} |C_\Gamma(x_i)[q]| \cdot S_i.
  \end{equation*}

  If $p = 2$, then each $S_i$ is zero. This implies that $\nu_n(\mathcal{C})$ is divisible by $2^k$. By the Chinese remainder theorem, $\nu_n(\mathcal{C})$ is divisible by $n = 2^k \cdot q$.

  If $p$ is odd, then by Lemma~\ref{lem:quadratic-sum}, each $S_i$ is divisible by $\sqrt{p}$. This implies that $\nu_n(\mathcal{C})$ is divisible by $p^{k - 1} \sqrt{p}$. Again by the Chinese remainder theorem, $\nu_n(\mathcal{C})$ is divisible by $q \cdot p^{k - 1} \sqrt{p} = n / \sqrt{p}$.
\end{proof}

As a direct consequence, we have the following corollaries.

\begin{corollary}
  \label{cor:Frobenius-2}
  Frobenius theorem holds for $\mathcal{C}$ if $c(\omega) = 2$.
\end{corollary}

\begin{corollary}
  Suppose that $c(\omega)$ is a prime number and that $\nu_n(\mathcal{C})$ is an integer for every $n$. Then Frobenius theorem holds for $\mathcal{C}$.
\end{corollary}

Theorem~\ref{thm:Frobenius-hopf} (d) follows from Corollary~\ref{cor:Frobenius-2}, see Remark~\ref{rem:suzuki-c-1} and Remark~\ref{rem:suzuki-c-2}. Recall that we gave some estimations of $c(\omega)$ in Section~\ref{sec:group-theoretical}. Theorem~\ref{thm:Frobenius-hopf}~(b) follows from the following corollary of Theorem~\ref{thm:Frobenius-1}, which can be obtained by a similar way as the proof of Corollary~\ref{cor:GT-int-criterion-2}.

\begin{corollary}
  \label{cor:Frobenius-3}
  Frobenius theorem holds for semisimple Hopf algebras $H$ fitting into an abelian extension $1 \to \mathbb{C}^G \to H \to \mathbb{C}\mathbb{Z}_2 \to 1$.
\end{corollary}

\section*{acknowledgements}
The author is grateful to Professor Mitsuhiro Takeuchi for his valuable comments on Section~\ref{sec:preliminaries} and Section~\ref{sec:regular}. This work is supported by Grant-in-Aid for JSPS Fellows ($21 \cdot 2147$).  


\end{document}